\documentclass[reqno,11pt]{amsart}

\usepackage[a4paper,textwidth=16.7cm,textheight=23.4cm,centering]{geometry}
\usepackage{amssymb,mathtools}
\usepackage[unicode,colorlinks=true,linkcolor=blue,citecolor=red,urlcolor=blue]{hyperref}
\usepackage{microtype}

\numberwithin{equation}{section}
\mathtoolsset{showonlyrefs}
\allowdisplaybreaks[1]
\makeatletter
\renewcommand*\l@subsection{\@tocline{2}{0pt}{2pc}{3pc}{}}
\makeatother

\newtheorem{theorem}{Theorem}[section]
\newtheorem{proposition}[theorem]{Proposition}
\newtheorem{lemma}[theorem]{Lemma}
\theoremstyle{definition}
\newtheorem{definition}[theorem]{Definition}
\theoremstyle{remark}
\newtheorem{remark}[theorem]{Remark}

\newcommand{\R}{\mathbb R}
\newcommand{\dd}{\mathop{}\!\mathrm{d}}

\title[Propagation of energy moments for the plasma--charge model]{\small Propagation of Energy Moments for the Plasma--Charge Model}
\author{Zhaopeng Wang}
\thanks{School of Mathematics and Statistics, Beijing Institute of Technology, Beijing 100081, China.\newline Email: \texttt{zpwang@bit.edu.cn}}
\hypersetup{
  pdftitle={Propagation of Energy Moments for the Plasma-Charge Model},
  pdfauthor={Zhaopeng Wang},
  pdfsubject={Vlasov-Poisson equations with a moving point charge},
  pdfkeywords={Vlasov-Poisson; plasma-charge; energy moments; convex transport; energy conservation}
}

\begin{document}

\begin{abstract}
This paper proves propagation of energy moments for the three-dimensional plasma--charge model, namely, the repulsive Vlasov--Poisson system with a moving point charge. For every $K>2$, an initially finite $K$th energy moment remains finite on every bounded time interval, without requiring any initial moment of order greater than $K$, a smallness condition on the plasma mass, or positive initial separation from the point charge. The constructed global weak solution conserves total energy for almost every time. The proof relies on a critical space--time density estimate and a direct control of the electric field along the moving charge.
\end{abstract}

\maketitle

\begingroup
\small
\noindent\textit{Keywords.} Vlasov--Poisson system, plasma--charge model, propagation of moments, convex transport, energy conservation.\par
\noindent\textit{2020 Mathematics Subject Classification.} Primary 35Q83. Secondary 35A01, 35B45, 82D10.
\endgroup

\begingroup
\small
\linespread{1.03}\selectfont
\hypersetup{linkcolor=black}
\tableofcontents
\endgroup
\clearpage

\section{Introduction}\label{sec:introduction}

\subsection{Physical background}\label{subsec:physical-background}

A plasma consists of a large number of charged particles that interact through electric fields. When collisions between particles can be neglected, the Vlasov--Poisson system is a typical kinetic model for their motion. Instead of following each particle, we use a distribution function $f(t,x,v)$ to describe how particles are distributed in position and velocity. The Vlasov equation describes their motion, while the Poisson equation determines the electric field generated by the plasma.

The plasma--charge model adds a distinguished charged particle to this description. Unlike the particles represented by the distribution function, this particle is followed individually. Its position and velocity are unknown and evolve together with the plasma. The distinguished particle produces a Coulomb field that acts on the plasma, while the plasma produces an electric field that accelerates the distinguished particle. Thus, the two parts of the system influence each other at every time.

We consider the three-dimensional repulsive case, in which the plasma particles and the point charge have the same sign. For simplicity, we normalize the physical constants to one. Let $f(t,x,v)\geq0$ be the plasma distribution, and let $\xi(t)$ and $\eta(t)$ denote the position and velocity of the point charge. The system reads
\begin{equation}\label{eq:system}
\begin{cases}
\partial_t f+v\cdot\nabla_x f+(E+F)\cdot\nabla_v f=0,\qquad f|_{t=0}=f_0,\qquad (t,x,v) \in \R^+\times\R^3\times\R^3,\\[3pt]
\displaystyle E(t,x)=\int_{\mathbb R^3}\frac{x-y}{|x-y|^3}\rho(t,y)\dd y,\quad
\rho(t,x)=\int_{\mathbb R^3}f(t,x,v)\dd v,\quad
F(t,x)=\frac{x-\xi(t)}{|x-\xi(t)|^3},\\[3pt]
\dot\xi(t)=\eta(t),\quad \dot\eta(t)=E(t,\xi(t)),\quad (\xi,\eta)|_{t=0}=(\xi_0,\eta_0).
\end{cases}
\end{equation}
Here $\rho$ is the spatial density of the plasma, $E$ is the electric field generated by the plasma, and $F$ is the Coulomb field generated by the moving point charge.

This model describes the interaction between a diffuse cloud of charges and a single moving point charge. The point charge may be viewed as an idealized tagged particle in the plasma. The essential feature is that its motion is not prescribed in advance. In particular, the Coulomb singularity is not fixed in space. It moves with the point charge and is coupled to the evolution of the plasma. This makes the problem different from a Vlasov equation with a given external force.

The point charge is accelerated only by the field $E$ generated by the plasma. Its own Coulomb field does not act on itself. Conversely, the point charge exerts the force $F$ on the plasma. The total force exerted by the point charge on the plasma is equal and opposite to the force exerted by the plasma on the point charge.

A basic property of the repulsive system is the conservation of total energy. For a smooth solution, the total energy is given by
\begin{equation}\label{eq:total-energy-intro}
\begin{aligned}
\mathcal E(t)={}&\frac12\iint_{\mathbb R^6}|v|^2 f(t,x,v)\dd x\dd v+\frac12|\eta(t)|^2\\
&+\frac12\iint_{\mathbb R^3\times\mathbb R^3}\frac{\rho(t,x)\rho(t,y)}{|x-y|}\dd x\,dy
+\int_{\mathbb R^3}\frac{\rho(t,x)}{|x-\xi(t)|}\dd x.
\end{aligned}
\end{equation}
It satisfies $\mathcal E(t)=\mathcal E(0)$. The four terms represent the kinetic energy of the plasma, the kinetic energy of the point charge, the interaction energy within the plasma, and the interaction energy between the plasma and the point charge. In the repulsive case, all four terms are nonnegative. Thus, the conservation law provides a uniform bound on each part of the energy.

The conservation of energy has a simple physical meaning. A large amount of plasma cannot move at very high velocities without carrying a large amount of kinetic energy. Thus, finite energy limits the total mass of high-velocity particles. However, it does not prevent a small number of particles from reaching very high velocities. These particles may have little effect on the total kinetic energy but a large effect on higher velocity moments. Consequently, energy conservation alone is not enough to control the long-time behavior of these moments.

The moving point charge brings another difficulty. The interaction energy controls the integral with the weight $|x-\xi(t)|^{-1}$, but the Coulomb force has the stronger singularity $|x-\xi(t)|^{-2}$. Moreover, the equation for the point charge requires the value of the plasma field at the moving point $x=\xi(t)$. An integral estimate for the field in space does not, by itself, control its value along this trajectory. We therefore need estimates that use both the motion of the plasma and the motion of the point charge.

There is nevertheless a useful structure in the relative motion. Let $(X(t),V(t))$ be a plasma characteristic, and set
\[
Y(t)=X(t)-\xi(t),\qquad U(t)=V(t)-\eta(t).
\]
Along this characteristic,
\begin{equation*}
\frac{\dd}{\dd t}\left(\frac12|U(t)|^2+\frac1{|Y(t)|}\right)
=U(t)\cdot\bigl(E(t,X(t))-E(t,\xi(t))\bigr).
\end{equation*}
The singular Coulomb force disappears from the right-hand side. This cancellation suggests measuring particles by their relative energy rather than by velocity alone.

Following the standard pointwise-energy approach for the plasma--charge model~\cite{CaprinoMarchioro2010,MMP2011}, we introduce
\[
a(t,x,v)=\left(1+|v-\eta(t)|^2+\frac{2}{|x-\xi(t)|}\right)^{1/2},
\]
and define the relative energy moments by
\begin{equation}\label{eq:moment-intro}
H_K(t)=\iint_{\mathbb R^6}a(t,x,v)^Kf(t,x,v)\dd x\dd v.
\end{equation}
This weight measures both large relative velocities and the concentration of plasma near the moving charge. In particular, the condition $H_K(0)<\infty$ does not require the plasma to vanish in a neighborhood of the initial charge position. The plasma may be present arbitrarily close to the charge, provided that the weighted integral is finite.

Our aim is to understand how far these energy bounds can be extended. In particular, we want to control the long-time growth of a prescribed relative energy moment without assuming moments of higher order. We also seek estimates that remain valid for arbitrary finite plasma mass and without an initial vacuum region around the point charge. As we shall see, the repulsive coupling provides additional structure that allows us to go beyond the bounds given by energy conservation alone.

\subsection{Previous results}\label{subsec:previous-results}

The Vlasov--Poisson system without a point charge has a well-developed global theory. Lions and Perthame~\cite{LionsPerthame1991} established global solutions and studied the propagation of velocity moments by an Eulerian method. Pfaffelmoser~\cite{Pfaffelmoser1992} and Schaeffer~\cite{Schaeffer1991} proved global existence of classical solutions for general smooth initial data by studying particle trajectories. These works form the basis of the global theory for the three-dimensional Vlasov--Poisson equation. For a broader account of the subject, we refer to Rein's survey~\cite{Rein2007}. Related questions have also been studied for ionic Vlasov--Poisson systems, which describe ions in a plasma with thermalized electrons. Griffin-Pickering and Iacobelli~\cite{GriffinPickeringIacobelli2021KRM} established global strong solutions for two such systems in $\mathbb R^3$. We also refer to their survey on quasineutral limits~\cite{GriffinPickeringIacobelli2021Survey} for further background on ionic models.

The propagation of velocity moments is an important part of this theory. Although kinetic energy controls the second velocity moment, the higher moments require additional estimates. Lions and Perthame developed a method for propagating moments of sufficiently high order. Pallard~\cite{Pallard2012} later proved that, in the whole space $\mathbb R^3$, every initially finite velocity moment of order $k>2$ can be propagated. In the periodic case, his result required $k>14/3$. Chen and Chen~\cite{ChenChen2019} subsequently improved the periodic threshold to $k>3$.  

The presence of a moving point charge makes the problem more difficult. Caprino and Marchioro~\cite{CaprinoMarchioro2010} initiated the mathematical study of the plasma--charge model in two dimensions. Marchioro, Miot and Pulvirenti~\cite{MMP2011} subsequently established global classical solutions in three dimensions for smooth initial data whose support is separated from the point charge. Their result shows that the Coulomb singularity does not prevent global classical motion when the plasma is initially separated from the charge. However, the separation condition excludes initial distributions that extend all the way to the charge.

An important step toward general initial data was made by Desvillettes, Miot and Saffirio~\cite{DMS2015}. They constructed global weak solutions and proved polynomial propagation of energy moments for the three-dimensional plasma--charge model. Their result allows the plasma to approach the point charge, but requires that its total mass lie below a threshold $\lambda\le1$. It assumes finite initial moments of every order $m<m_0$ for some $m_0>6$ and propagates moments of order $m<\min\{m_0,7\}$. Their work extended the moment method of Lions and Perthame to a system with a moving Coulomb singularity.

Further developments weakened some of these restrictions and extended the classes of solutions. Li and Zhang~\cite{LiZhang2017} studied solutions with unbounded supports and the propagation of mixed velocity--space moments. Their work also treated certain classical solutions with infinite kinetic energy. Crippa, Ligabue and Saffirio~\cite{CrippaLigabueSaffirio2018} developed a Lagrangian theory for the three-dimensional plasma--charge system. Under a smallness condition on the total diffuse charge and suitable decay assumptions near the point charge, they constructed global Lagrangian solutions and showed that these solutions are transported by particle trajectories.

The small-mass restriction in the moment theory was later removed by Wu and Zhang~\cite{WuZhang2021}. In their work on large data, they established global weak solutions and polynomial propagation of moments without assuming that the plasma mass is small. Their Theorem~1.1 retains the initial moment hypothesis for orders $m<m_0$ with $m_0>6$ and propagates every such moment. The remaining question for moment propagation is whether one can start from a single prescribed moment just above the energy level, without assuming any higher moment. 

There are also results on the long-time behavior of small solutions. Pausader and Widmayer~\cite{PausaderWidmayer2021} studied the stability of a point charge under radial symmetry and obtained modified scattering. Pausader, Widmayer and Yang~\cite{PausaderWidmayerYang2026} extended the stability analysis to a more general setting without the same radial restriction. These works give a detailed description of the asymptotic motion of the plasma and the point charge under suitable smallness and regularity assumptions. They concern a different regime from the large-data moment problem considered here.

A different line of development comes from the compensated-integrability theory of Serre~\cite{Serre2018,Serre2019}. This theory gives improved integrability for positive symmetric tensors whose divergence is controlled. For kinetic equations, the mass and momentum densities naturally form such a tensor. Its positivity reflects the fact that particles carry nonnegative mass, while its determinant contains information about the spatial density. This provides a way to obtain density estimates from the structure of the transport equation, rather than from a direct estimate of the electric field alone.

In a recent work on the classical Vlasov--Poisson system~\cite{WangConservation2026}, the author used this method to prove a critical space-time density estimate in three dimensions, both in the whole space and on the torus. In the bounded finite-energy weak-solution class considered there, the estimate gives conservation of total energy. In the periodic repulsive case, it also gives propagation of every velocity moment of order $k>2$.

The plasma--charge system presents a new difficulty for this approach. The point charge produces an additional singular force, and its position is itself an unknown of the equation. If one applies the density estimate only to the diffuse plasma, this force appears as a singular source that must be controlled separately. However, the force acting on the plasma and the force acting on the point charge come from the same interaction. In the repulsive case, this relation has a useful sign when the two equations are treated together. This suggests using the full plasma--charge coupling, rather than regarding the point charge as an external source.

In this paper, we apply the density method to the coupled plasma--charge system. The interaction with the point charge requires estimates that use the motion of the charge as well as the plasma equation. We prove propagation of one prescribed relative energy moment of order $K>2$, together with bounds for the density and the cumulative singular interactions.  

\subsection{Main results}\label{subsec:main-results}

We now state our main result. As discussed above, the main difficulty is to control the higher relative energy moments in the presence of a moving Coulomb singularity. We shall show that every prescribed moment of order $K>2$ can be propagated without assuming any higher initial moment. Moreover, the growth of this moment can be bounded by a polynomial in time. The proof also gives global space-time estimates for the plasma density and quantitative bounds for the cumulative singular interactions.

Since our initial data need not vanish near the point charge, we work with weak solutions. The singular force requires some care. In particular, the electric field along the point-charge trajectory must be well defined, and the singular force in the Vlasov equation must be integrable. We use the following definition.

\begin{definition}\label{def:weak-solution}
Fix $T>0$. A weak solution of \eqref{eq:system} on $[0,T]$ is a triple $(f,\xi,\eta)$ satisfying
\[
f\ge0,\qquad f\in L^\infty(0,T;L^1(\mathbb R^6)\cap L^\infty(\mathbb R^6)),
\]
\[
\xi\in C^1([0,T];\mathbb R^3),\qquad
\eta\in C([0,T];\mathbb R^3)\cap W^{1,1}(0,T;\mathbb R^3),
\]
with $\xi(0)=\xi_0$, $\eta(0)=\eta_0$, and $\dot\xi=\eta$. The density $\rho$ and the field $E$ are defined as in \eqref{eq:system}. We require
\begin{equation}\label{eq:weak-singular-integrability}
\int_0^T\int_{\mathbb R^3}\frac{\rho(t,x)}{|x-\xi(t)|^2}\dd x\,dt<\infty.
\end{equation}
For every $\varphi\in C_c^\infty([0,T)\times\mathbb R^6)$, the Vlasov equation holds in the form
\begin{equation}\label{eq:weak-vlasov}
\begin{aligned}
&\int_0^T\iint_{\mathbb R^6}
f\left(\partial_t\varphi+v\cdot\nabla_x\varphi+(E+F)\cdot\nabla_v\varphi\right)\dd x\dd v\dd t\\
&\qquad+\iint_{\mathbb R^6}f_0(x,v)\varphi(0,x,v)\dd x\dd v=0.
\end{aligned}
\end{equation}
Moreover, the Coulomb integral defining $E(t,\xi(t))$ converges absolutely for almost every $t\in(0,T)$, and
\[
\eta(t)=\eta_0+\int_0^t E(s,\xi(s))\dd s,\qquad 0\le t\le T.
\]
A global weak solution satisfies these conditions on every bounded time interval.
\end{definition}

We can now state the main theorem.  

\begin{theorem}\label{thm:main}
Let $K>2$, let $(\xi_0,\eta_0)\in\mathbb R^3\times\mathbb R^3$, and assume that
\[
0\le f_0\in L^1(\mathbb R^6)\cap L^\infty(\mathbb R^6),\qquad H_K(0)<\infty.
\]
Then the plasma--charge system \eqref{eq:system} has a global weak solution $(f,\xi,\eta)$ in the sense of Definition~\ref{def:weak-solution}. This solution conserves total energy:
\begin{equation}\label{eq:main-energy-conservation}
\mathcal E(t)=\mathcal E(0)\qquad\text{for almost every }t\ge0.
\end{equation}
Let $C_0$ depend only on the initial mass, the initial $L^\infty$ norm, and the initial total energy. The constant $C_K$ may also depend on $K$ and $H_K(0)$. The solution satisfies the following estimates.
\begin{enumerate}
\renewcommand{\labelenumi}{\textup{(\roman{enumi})}}
\setlength{\itemsep}{2pt}
\item \emph{Global space-time density estimate.}
\begin{equation}\label{eq:main-global-density}
\int_0^\infty\left(\|\rho(t)\|_{L^2(\R^3)}^2+\|E(t)\|_{L^6(\R^3)}^2\right)\,\dd t\le C_0.
\end{equation}
\item \emph{Cumulative singular interactions.} For every $T\ge0$,
\begin{equation}\label{eq:main-cumulative-interaction}
\int_0^T|E(t,\xi(t))|\,\dd t+
\int_0^T\int_{\R^3}\frac{\rho(t,x)}{|x-\xi(t)|^2}\,\dd x\dd t\le C_0(1+T)^{2/3}.
\end{equation}
\item \emph{Polynomial growth of energy moments.} For every $T>0$,
\begin{equation}\label{eq:main-polynomial-moment}
\operatorname*{ess\,sup}_{0<t<T}(1+H_K(t))\le C_K(1+T)^{K-2}.
\end{equation}
\end{enumerate}

The charge velocity also satisfies the H\"older estimate
\begin{equation}\label{eq:main-holder}
|\eta(t)-\eta(s)|\le C_{T,K}|t-s|^{\alpha_K},\qquad
\alpha_K=\frac{3(K-1)}{3K+1},\qquad 0\le s\le t\le T.
\end{equation}
In particular, $\eta\in C_{\mathrm{loc}}^{0,\alpha_K}([0,\infty))$ and $\xi\in C_{\mathrm{loc}}^{1,\alpha_K}([0,\infty))$. If $K\ge4$, the charge velocity is locally Lipschitz.
\end{theorem}

\begin{remark}\label{rem:single-moment}
The order $K>2$ is fixed in the theorem. We assume only that $H_K(0)$ is finite. No moment of order higher than $K$ is required. There is no smallness assumption on the plasma mass and no initial vacuum condition around the point charge. The theorem gives existence of a global weak solution, but does not assert uniqueness.
\end{remark}

\begin{remark}\label{rem:lower-order-growth}
For $2<K\le4$, Proposition~\ref{prop:low-moments} gives the stronger growth exponent $\frac23(K-2)$. The bound passes to the weak solution by Proposition~\ref{prop:limit-moments}. Thus
\[
\operatorname*{ess\,sup}_{0<t<T}(1+H_K(t))\le C_K(1+T)^{\frac23(K-2)}.
\]
In particular, if $H_4(0)<\infty$, the essential supremum of $H_4$ on $(0,T)$ is bounded by $C_4(1+T)^{4/3}$. The second moment is controlled by the total energy.
\end{remark}

\subsection{Ideas of the proof}\label{subsec:ideas-of-proof}

The starting point is a global space--time estimate for the plasma density. The main difficulty is the singular force generated by the point charge. A direct use of the density method for the Vlasov--Poisson system would leave this force as an additional term. The plasma and the point charge, however, exert equal and opposite forces on each other. We use this cancellation by testing the two equations together with a function $\varphi(t,x)$ that is jointly convex in $(t,x)$. The interaction between the plasma and the charge then appears as
\[
I_{pc}(t)=\int_{\mathbb R^3}\rho(t,x)
\frac{(x-\xi(t))\cdot\bigl(\nabla_x\varphi(t,x)-\nabla_x\varphi(t,\xi(t))\bigr)}{|x-\xi(t)|^3}\,\dd x\ge0.
\]
The same sign holds for the interaction among plasma particles. This positivity is the main reason why the singular force can be handled through the coupling of the two equations. The kinetic tensor estimate and Brenier's optimal transport theorem~\cite{Brenier1991}, together with the density method developed for the classical Vlasov--Poisson system~\cite{Serre2018,Serre2019,WangConservation2026}, give
\[
\int_0^\infty\left(\|\rho(t)\|_{L^2(\mathbb R^3)}^2+\|E(t)\|_{L^6(\mathbb R^3)}^2\right)\,\dd t\le C.
\]

This density estimate does not control the electric field at the point charge. The relative motion provides another useful identity. Let $(X,V)$ be a plasma characteristic, and set $Y=X-\xi$, $U=V-\eta$, $e=Y/|Y|$, and $p=U\cdot e$. Along the characteristic,
\begin{equation*}
\dot p=\frac{|U|^2-p^2}{|Y|}+\frac1{|Y|^2}
+\bigl(E(t,X)-E(t,\xi(t))\bigr)\cdot e.
\end{equation*}
The first two terms are nonnegative. After integration in time, the identity controls the accumulated singular interaction by the endpoint relative velocities and the fields along the two trajectories. To estimate the field at the charge, we split particles on a short interval according to their relative velocities. The low-velocity group is controlled by the density estimate, while the high-velocity group has small mass. This gives a local field bound. We choose the interval length according to its density. If
\[
B_I=\int_I\|\rho(t)\|_2^2\,\dd t,\qquad
Q_I=\int_I|E(t,\xi(t))|\,\dd t,
\]
then the local estimate has the form
\[
|I|^2B_I\le c_0\quad\Longrightarrow\quad Q_I\le L_0.
\]
The global density bound controls the number of such intervals. Summing the local estimates and using the radial identity gives
\[
Q_T+D_T\le C(1+T)^{2/3},\qquad
D_T=\int_0^T\int_{\mathbb R^3}\frac{\rho(t,x)}{|x-\xi(t)|^2}\,\dd x\dd t.
\]

The propagation of higher relative energy moments requires control of the field along individual plasma characteristics. For a characteristic labeled by $z$, write
\[
G_T(z)=\int_0^T|E(t,X_z(t))|\,\dd t,\qquad
Q_T=\int_0^T|E(t,\xi(t))|\,\dd t.
\]
The relative energy identity gives $a(t,X_z(t),V_z(t))\le a_0(z)+G_T(z)+Q_T$. The term $Q_T$ is already controlled. We estimate $G_T$ by considering particles for which it is large and at least as large as their initial relative energy. For $\lambda\ge\max\{1,Q_T\}$, set
\[
A_\lambda=\{z:G_T(z)>\lambda,\ G_T(z)\ge a_0(z)\}.
\]
This set consists of fixed initial labels, and its distribution satisfies a homogeneous transport equation. The convex transport estimate and the radial identity give
\[
\mu(A_\lambda)\le C(1+T)^3\lambda^{-15/2},\qquad
\dd\mu(z)=f_0(z)\,\dd z.
\]
Integrating this tail bound yields, for $2<K\le4$,
\[
\sup_{0\le t\le T}H_K(t)\le C_K(1+T)^{\frac23(K-2)}.
\]
In particular, the fourth relative energy moment grows at most like $(1+T)^{4/3}$.

For higher moments, a fixed group of initial labels is no longer sufficient. We divide the plasma into dyadic layers according to its current relative energy. Let $\chi\in C_c^\infty((1/2,4))$ satisfy $0\le\chi\le1$ and $\chi=1$ on $[1,2]$, and put
\[
g_R(t,x,v)=f(t,x,v)\chi\!\left(\frac{a(t,x,v)}R\right)^2.
\]
Since particles may enter or leave a layer, $g_R$ satisfies a transport equation with a source term. The square cutoff gives
\[
\int_{\mathbb R^3}|s_R(t,x,v)|\,\dd v
\le\frac C R\bigl(|E(t,x)|+|E(t,\xi(t))|\bigr)
\sqrt{\rho_R(t,x)\sigma_R(t,x)},
\]
where $\rho_R$ is the density of the layer and $\sigma_R$ is the density of a slightly larger layer. This source can be included in the convex transport estimate for the same layer and absorbed. The mass $m_R$ of the enlarged layer satisfies
\[
m_R\le C\min\left\{R^{-2},(1+T)^{4/3}R^{-4}\right\},
\]
by the second- and fourth-moment bounds. Applying the convex transport estimate to each layer and summing the field contributions over dyadic $R$ gives
\[
\int_0^T\|E(t)\|_{L^\infty(\mathbb R^3)}\,\dd t\le C(1+T).
\]
Thus the relative energy along every plasma characteristic is bounded by its initial value plus $C(1+T)$. Combining this estimate with the uniform second moment gives the growth rate $(1+T)^{K-2}$ for $K>4$. Together with the estimate for $2<K\le4$, this proves the moment bound for every $K>2$.

We construct the weak solution by approximation with smooth, compactly supported data separated from the point charge. The H\"older bound for the charge velocity gives compactness of its trajectory, while velocity averaging gives strong convergence of the densities and electric fields. The singular forces require an additional estimate near the moving charge. In particular, the approximating densities satisfy
\[
\sup_n\int_0^T\int_{|x-\xi_n(t)|<\varepsilon}
\frac{\rho_n(t,x)}{|x-\xi_n(t)|^2}\,\dd x\dd t
\le C_{T,K}\varepsilon^{1/2}.
\]
This estimate allows us to remove a cutoff from the Coulomb kernel after passing to the limit. The limiting functions then satisfy both the Vlasov equation and the equation of motion of the point charge. The moment bounds control the velocity tails, and the strong convergence of the fields together with the estimate near the charge allows each of the four energy terms to pass to the limit. The resulting global weak solution conserves total energy.

Section~\ref{sec:preliminaries} gives the basic estimates and the convex transport principle. Section~\ref{sec:global-density-interactions} proves the global density estimate and controls the cumulative singular interactions. Section~\ref{sec:moment-growth} establishes the growth of the relative energy moments. Section~\ref{sec:weak} constructs global weak solutions and proves conservation of total energy.

\section{Preliminaries and convex transport}
\label{sec:preliminaries}

\subsection{Conservation laws and relative motion}
\label{subsec:conservation}

We begin with the basic properties of the plasma--charge system. They follow from the conservation laws and the motion of particles relative to the point charge. The resulting estimates will be used throughout the paper.

We consider a global classical solution of \eqref{eq:system} with smooth, compactly supported initial data whose spatial support is separated from the point charge. All estimates below are independent of this separation distance. Let $\mathcal E_0=\mathcal E(0)$ be the initial total energy. We write $C(f_0)$ for a constant depending only on $\|f_0\|_{L^1(\mathbb R^6)}$ and $\|f_0\|_{L^\infty(\mathbb R^6)}$, and $C(f_0,\mathcal E_0)$ when the initial energy may also enter. Such constants may change from line to line.

\begin{lemma}\label{lem:basic-energy}
For every $t\geq0$, the solution satisfies
\[
\|f(t)\|_{L^1(\mathbb R^6)}=\|f_0\|_{L^1(\mathbb R^6)},\qquad
\|f(t)\|_{L^\infty(\mathbb R^6)}=\|f_0\|_{L^\infty(\mathbb R^6)},\qquad
\mathcal E(t)=\mathcal E_0.
\]
The uniform energy bounds take the form
\begin{equation}\label{eq:finite-energy-bounds}
\sup_{t\geq0}\bigl(H_2(t)+|\eta(t)|^2+\|E(t)\|_{L^2(\mathbb R^3)}^2\bigr)
\leq C(f_0,\mathcal E_0),
\end{equation}
In addition, for $1\leq p\leq5/3$,
\[
\sup_{t\geq0}\|\rho(t)\|_{L^p(\mathbb R^3)}\leq C(f_0,\mathcal E_0),
\]
and the electric field satisfies
\[
\sup_{t\geq0}\int_{\mathbb R^3}\rho(t,x)|E(t,x)|\,dx\leq C(f_0,\mathcal E_0).
\]
\end{lemma}

\begin{proof}
The characteristics satisfy $\dot X=V$ and $\dot V=E(t,X)+F(t,X)$. The vector field $(v,E+F)$ has zero divergence in phase space. Its flow therefore preserves volume, and $f$ is constant along characteristics. This gives conservation of the mass and the $L^\infty$ norm.

For the total energy, set $j(t,x)=\int_{\mathbb R^3}vf(t,x,v)\,dv$ and
$U_\rho(t,x)=\int_{\mathbb R^3}\frac{\rho(t,y)}{|x-y|}\,dy$.
The Vlasov equation gives $\partial_t\rho+\operatorname{div}_xj=0$. Since $E=-\nabla_xU_\rho$, the two kinetic energies satisfy
\[
\frac{d}{dt}\left(\frac12\iint_{\mathbb R^6}|v|^2f\,dx\,dv+\frac12|\eta|^2\right)
=\int_{\mathbb R^3}j\cdot(E+F)\,dx+\eta\cdot E(t,\xi).
\]
On the other hand,
\[
\frac{d}{dt}\left(\frac12\int_{\mathbb R^3}\rho U_\rho\,dx\right)
=-\int_{\mathbb R^3}j\cdot E\,dx
\]
and
\[
\frac{d}{dt}\int_{\mathbb R^3}\frac{\rho(t,x)}{|x-\xi(t)|}\,dx
=-\int_{\mathbb R^3}j\cdot F\,dx-\eta\cdot E(t,\xi).
\]
All force terms cancel when these identities are added. Thus $\mathcal E(t)=\mathcal E_0$. We also have
\[
\frac12\int_{\mathbb R^3}\rho U_\rho\,dx
=\frac1{8\pi}\int_{\mathbb R^3}|E|^2\,dx.
\]
Every term in the total energy is nonnegative. In particular, the plasma kinetic energy, $|\eta|^2$, $\|E\|_2^2$, and the plasma--charge interaction energy are uniformly bounded. Since $a(t,x,v)^2=1+|v-\eta(t)|^2+2/|x-\xi(t)|$ and
$|v-\eta|^2\leq2|v|^2+2|\eta|^2$, the second relative energy moment is uniformly bounded as well.

For the spatial density, split the velocity integral at radius $R>0$:
\[
\rho(t,x)\leq C\|f_0\|_{L^\infty}R^3
+R^{-2}\int_{\mathbb R^3}|v|^2f(t,x,v)\,dv.
\]
Optimizing in $R$ yields
\[
\rho(t,x)^{5/3}\leq C\|f_0\|_{L^\infty}^{2/3}
\int_{\mathbb R^3}|v|^2f(t,x,v)\,dv.
\]
Integration in $x$ proves the uniform $L^{5/3}$ bound for $\rho$. The bounds for $1\leq p\leq5/3$ follow by interpolation with the conserved mass.

Finally, the Hardy--Littlewood--Sobolev inequality gives
$\|E(t)\|_{L^{15/4}}\leq C\|\rho(t)\|_{L^{5/3}}$. Also,
\[
\|\rho(t)\|_{L^{15/11}}\leq
\|\rho(t)\|_{L^1}^{1/3}\|\rho(t)\|_{L^{5/3}}^{2/3}.
\]
Since $15/11$ and $15/4$ are conjugate, H\"older's inequality gives
\[
\int_{\mathbb R^3}\rho|E|\,dx
\leq\|\rho\|_{L^{15/11}}\|E\|_{L^{15/4}}
\leq C(f_0,\mathcal E_0).
\]
In particular, for every bounded interval $I$,
\begin{equation}\label{eq:rho-E-interval}
\int_I\int_{\mathbb R^3}\rho|E|\,dx\,dt
\leq C(f_0,\mathcal E_0)|I|.
\end{equation}
\end{proof}

The next lemma concerns motion relative to the point charge. The singular Coulomb force has two different effects: it cancels in the relative energy identity, but contributes a positive term to the derivative of the radial velocity.

\begin{lemma}\label{lem:relative-identities}
Let $(X(t),V(t))$ be a plasma characteristic. Set $Y=X-\xi$, $U=V-\eta$, $e=Y/|Y|$, and $p=U\cdot e$. The relative energy satisfies
\begin{equation}\label{eq:relative-energy}
\frac{d}{dt}\left(\frac12|U|^2+\frac1{|Y|}\right)
=U\cdot\bigl(E(t,X)-E(t,\xi)\bigr)
\end{equation}
The radial velocity satisfies
\begin{equation}\label{eq:radial-identity}
\dot p=\frac{|U|^2-p^2}{|Y|}+\frac1{|Y|^2}
+\bigl(E(t,X)-E(t,\xi)\bigr)\cdot e.
\end{equation}
\end{lemma}

\begin{proof}
The relative trajectories satisfy $\dot Y=U$ and
$\dot U=E(t,X)-E(t,\xi)+Y/|Y|^3$. Therefore,
\[
\frac{d}{dt}\frac12|U|^2=U\cdot\bigl(E(t,X)-E(t,\xi)\bigr)+\frac{U\cdot Y}{|Y|^3},\qquad
\frac{d}{dt}\frac1{|Y|}=-\frac{U\cdot Y}{|Y|^3}.
\]
The singular terms cancel in the sum. Moreover, the definition of $a$ yields
\begin{equation}\label{eq:relative-weight}
\left|\frac{d}{dt}a(t,X(t),V(t))\right|
\leq|E(t,X(t))|+|E(t,\xi(t))|.
\end{equation}

For the second identity, differentiate $e=Y/|Y|$ to obtain $\dot e=(U-pe)/|Y|$. Since $\dot p=\dot U\cdot e+U\cdot\dot e$, the radial identity follows. The term $(|U|^2-p^2)/|Y|$ is nonnegative because $|p|\leq|U|$.
\end{proof}

Two estimates follow immediately and will be used later. Integrating the radial identity over an interval $I=[t_0,t_1]$ gives
\begin{equation}\label{eq:radial-trajectory-bound}
\int_I\frac{dt}{|Y(t)|^2}\leq |U(t_0)|+|U(t_1)|
+\int_I\bigl(|E(t,X(t))|+|E(t,\xi(t))|\bigr)\,dt.
\end{equation}
We integrate this inequality over particle labels at time $t_0$, with measure $f(t_0,z)\,dz$. The characteristic flow preserves this measure. By Lemma~\ref{lem:basic-energy}, the two endpoint terms are bounded by
\[
\iint_{\mathbb R^6}|v-\eta(t_i)|f(t_i,x,v)\,dx\,dv
\leq \|f_0\|_{L^1}^{1/2}H_2(t_i)^{1/2}\leq C(f_0,\mathcal E_0),\qquad i=0,1.
\]
The plasma-field contribution is bounded by \eqref{eq:rho-E-interval}. The remaining field contribution is evaluated along the charge trajectory. Together these bounds yield
\begin{equation}\label{eq:radial-integrated-bound}
\int_I\int_{\mathbb R^3}\frac{\rho(t,x)}{|x-\xi(t)|^2}\,dx\,dt
\leq C(f_0,\mathcal E_0)(1+|I|)
+\|f_0\|_{L^1}\int_I|E(t,\xi(t))|\,dt.
\end{equation}

We finish with two estimates for Coulomb fields. The Hardy--Littlewood--Sobolev inequality gives, for $1<p<3$,
\begin{equation}\label{eq:coulomb-HLS}
\bigl\||\cdot|^{-2}*h\bigr\|_{L^q(\mathbb R^3)}
\leq C_p\|h\|_{L^p(\mathbb R^3)},\qquad
\frac1q=\frac1p-\frac13.
\end{equation}
We will use $(p,q)=(6/5,2)$, $(3/2,3)$, and $(2,6)$.

If $h\geq0$ belongs to $L^p(\mathbb R^3)\cap L^\infty(\mathbb R^3)$, then
\begin{equation}\label{eq:coulomb-pointwise}
\bigl\||\cdot|^{-2}*h\bigr\|_{L^\infty}
\leq C_p\|h\|_{L^\infty}^{1-p/3}\|h\|_{L^p}^{p/3}.
\end{equation}
Indeed, the integrals over $|x-y|\leq R$ and $|x-y|>R$ are bounded by
$CR\|h\|_\infty$ and $C_pR^{1-3/p}\|h\|_p$, respectively. The optimal choice of $R$ gives the stated interpolation estimate.

A related estimate is useful for subpopulations. Suppose that
$0\leq\sigma\leq\rho$ and $\sigma\in L^\infty(\mathbb R^3)$. Using $\sigma$ near $x$ and $\rho$ away from $x$ gives
\[
\int_{\mathbb R^3}\frac{\sigma(y)}{|x-y|^2}\,dy
\leq CR\|\sigma\|_\infty+CR^{-4/5}\|\rho\|_{L^{5/3}}.
\]
Optimizing in $R$ yields
\begin{equation}\label{eq:coulomb-subpopulation}
\sup_{x\in\mathbb R^3}\int_{\mathbb R^3}\frac{\sigma(y)}{|x-y|^2}\,dy
\leq C\|\rho\|_{L^{5/3}}^{5/9}\|\sigma\|_{L^\infty}^{4/9}.
\end{equation}
By Lemma~\ref{lem:basic-energy}, the first factor is uniformly bounded in time.

The finite-energy estimates above do not give the space-time $L^2$ bound needed below. We obtain this bound from the transport equation.

\subsection{A convex transport estimate}
\label{subsec:convex-transport}

The energy bounds in the previous subsection give $\rho\in L^\infty_tL^{5/3}_x$. To obtain the stronger space-time $L^2$ estimate, we test the transport equation with the derivative of a convex function along particle trajectories.

We need this estimate in two settings. For the full plasma--charge system, the plasma and the charge can be tested together; the repulsive forces then have a favorable sign. For a selected group of particles, the force need not have this sign, and the group may satisfy a transport equation with a source term. We use the positive kinetic tensor from Serre's compensated-integrability approach \cite{Serre2018,Serre2019}, together with the determinant bound in \cite{WangConservation2026} and Brenier's convex transport theorem \cite{Brenier1991}. We begin with an identity that covers both settings.

Let $I=[t_-,t_+]$ be a bounded interval. Let $g\geq0$ be bounded on $I\times\mathbb R^6$ and satisfy
$\int_I\iint_{\mathbb R^6}(1+|v|^2)g\,dx\,dv\,dt<\infty$.
Suppose that, in distributions,
\begin{equation}\label{eq:general-transport}
\partial_tg+v\cdot\nabla_xg+\operatorname{div}_v(bg)=s.
\end{equation}
Assume $bg\in L^1(I\times\mathbb R^6)$, $(1+|v|)s\in L^1(I\times\mathbb R^6)$, and that $g$ has strong traces $g_-$ and $g_+$ at $t_-$ and $t_+$ in $L^1(\mathbb R^6;(1+|v|)\,dx\,dv)$.

We also allow a distinguished trajectory with $\dot\xi=\eta$ and $\dot\eta=b_c$, where $\xi\in C^1(I;\mathbb R^3)$ and $\eta\in W^{1,1}(I;\mathbb R^3)$. A parameter $\theta\in\{0,1\}$ indicates whether the trajectory is included: $\theta=1$ for the joint test and $\theta=0$ for the kinetic equation alone.

For a smooth function $\varphi=\varphi(t,x)$, set
\[
\Psi_\varphi=\partial_t\varphi+v\cdot\nabla_x\varphi,\qquad
q_\varphi=(1,v)^TD_{t,x}^2\varphi(1,v).
\]
Here $\Psi_\varphi$ is the derivative of $\varphi$ along free transport, while $q_\varphi$ is nonnegative for convex $\varphi$.

\begin{lemma}\label{lem:convex-identity}
For every smooth $\varphi$ with bounded first and second derivatives,
\begin{equation}\label{eq:convex-identity}
\begin{aligned}
&\int_I\iint_{\mathbb R^6}q_\varphi g\,dx\,dv\,dt
+\theta\int_Iq_\varphi(t,\xi(t),\eta(t))\,dt+\mathcal F_\varphi\\
&\quad=\left[\iint_{\mathbb R^6}\Psi_\varphi g\,dx\,dv
+\theta\Psi_\varphi(t,\xi(t),\eta(t))\right]_{t_-}^{t_+}
-\int_I\iint_{\mathbb R^6}\Psi_\varphi s\,dx\,dv\,dt,
\end{aligned}
\end{equation}
where
\begin{equation}\label{eq:signed-force}
\mathcal F_\varphi
=\int_I\iint_{\mathbb R^6}bg\cdot\nabla_x\varphi\,dx\,dv\,dt
+\theta\int_Ib_c(t)\cdot\nabla_x\varphi(t,\xi(t))\,dt.
\end{equation}
All charge terms are omitted if $\theta=0$.
\end{lemma}

\begin{proof}
We have $\nabla_v\Psi_\varphi=\nabla_x\varphi$ and
$(\partial_t+v\cdot\nabla_x)\Psi_\varphi=q_\varphi$.
Multiplying \eqref{eq:general-transport} by $\Psi_\varphi$ and integrating by parts gives
\begin{equation*}
\begin{aligned}
\left[\iint_{\mathbb R^6}\Psi_\varphi g\,dx\,dv\right]_{t_-}^{t_+}
={}&\int_I\iint_{\mathbb R^6}q_\varphi g\,dx\,dv\,dt
+\int_I\iint_{\mathbb R^6}bg\cdot\nabla_x\varphi\,dx\,dv\,dt\\
&+\int_I\iint_{\mathbb R^6}\Psi_\varphi s\,dx\,dv\,dt.
\end{aligned}
\end{equation*}
To justify this calculation, use smooth cutoffs in $x$ and $v$ at radius $R$. The spatial cutoff error is bounded by
$C_\varphi R^{-1}\int_I\iint(1+|v|^2)g$, which tends to zero. The velocity cutoff error is controlled by the tail of $|b|g$, since $|\Psi_\varphi|/R$ is bounded on its support. The assumed integrability of $(1+|v|)s$ controls the source term. Temporal cutoffs and the strong endpoint traces give the two boundary terms.

Along the distinguished trajectory, the chain rule gives
\[
\frac{d}{dt}\Psi_\varphi(t,\xi(t),\eta(t))
=q_\varphi(t,\xi(t),\eta(t))
+b_c(t)\cdot\nabla_x\varphi(t,\xi(t)).
\]
Adding the integrated charge identity to the kinetic identity, with weight $\theta$, accounts for every term in the signed transport balance.
\end{proof}

For convex $\varphi$, both Hessian terms in \eqref{eq:convex-identity} are nonnegative. Thus we can control the Hessian action against $g$ whenever the negative part of the signed force is bounded. We now choose $\varphi$ by convex transport so that its Hessian controls $\rho_g^2$.

\begin{proposition}\label{prop:convex-density}
Let the assumptions of Lemma~\ref{lem:convex-identity} hold, and let
$\rho_g(t,x)=\int_{\mathbb R^3}g(t,x,v)\,dv$.
Fix $\ell>0$. Suppose that $\mathcal F_\varphi\geq-\mathcal B$ for every smooth convex $\varphi$ with bounded Hessian and $|\partial_t\varphi|\leq\ell$, $|\nabla_x\varphi|\leq1$, where $\mathcal B\geq0$ is independent of $\varphi$. Define
\begin{equation}\label{eq:budget-zero}
\mathcal A_0=\iint_{\mathbb R^6}(g_-+g_+)\,dx\,dv+2\theta
+\int_I\iint_{\mathbb R^6}|s|\,dx\,dv\,dt
\end{equation}
and
\begin{equation}\label{eq:budget-one}
\begin{aligned}
\mathcal A_1={}&\iint_{\mathbb R^6}|v|(g_-+g_+)\,dx\,dv
+\theta\bigl(|\eta(t_-)|+|\eta(t_+)|\bigr)\\
&+\int_I\iint_{\mathbb R^6}|v||s|\,dx\,dv\,dt+\mathcal B.
\end{aligned}
\end{equation}
Then
\begin{equation}\label{eq:convex-density-scale}
\|\rho_g\|_{L^2(I\times\mathbb R^3)}^{3/2}
\leq C\|g\|_{L^\infty(I\times\mathbb R^6)}^{1/2}\ell^{-1/4}
(\ell\mathcal A_0+\mathcal A_1).
\end{equation}
If the same $\mathcal B$ works for every $\ell>0$, then
\begin{equation}\label{eq:convex-density-optimized}
\|\rho_g\|_{L^2(I\times\mathbb R^3)}^2
\leq C\|g\|_{L^\infty(I\times\mathbb R^6)}^{2/3}
\mathcal A_0^{1/3}\mathcal A_1.
\end{equation}
\end{proposition}

\begin{proof}
If $g=0$ almost everywhere, there is nothing to prove. We therefore assume $\|g\|_\infty>0$.

For every convex test in the proposition, the charge Hessian term in \eqref{eq:convex-identity} is nonnegative. Since $|\Psi_\varphi|\leq\ell+|v|$, the signed force bound and the endpoint and source estimates give
\begin{equation}\label{eq:convex-hessian-budget}
\int_I\iint_{\mathbb R^6}q_\varphi g\,dx\,dv\,dt
\leq\ell\mathcal A_0+\mathcal A_1.
\end{equation}
Introduce the nonnegative symmetric kinetic tensor
\[
T_g(t,x)=\int_{\mathbb R^3}(1,v)\otimes(1,v)g(t,x,v)\,dv.
\]
Then \eqref{eq:convex-hessian-budget} becomes
\begin{equation}\label{eq:tensor-hessian-budget}
\int_I\int_{\mathbb R^3}\operatorname{tr}(T_gD_{t,x}^2\varphi)\,dx\,dt
\leq\ell\mathcal A_0+\mathcal A_1.
\end{equation}
We use the determinant estimate established in \cite[Lemma~3.3]{WangConservation2026}:
\begin{equation}\label{eq:tensor-det}
\det T_g(t,x)\geq c\|g\|_\infty^{-2}\rho_g(t,x)^6.
\end{equation}
The estimate applies almost everywhere because the local velocity density has finite mass and second moment there. For every nonnegative symmetric $4\times4$ matrix $H$, the arithmetic--geometric mean inequality gives
\[
\operatorname{tr}(T_gH)\geq4(\det T_g\det H)^{1/4}.
\]

For $N>0$, choose the convex potential as follows. Set
\[
w_N(t,x)=\mathbf1_{I\times B_N}(t,x)\min\{\rho_g(t,x)^2,N\},
\qquad S_N=\int_{\mathbb R^4}w_N\,dt\,dx.
\]
We take $N$ with $S_N>0$. Transport $w_N/S_N$ to the uniform probability measure on
$\Omega_\ell=(-\ell,\ell)\times B_1\subset\mathbb R^4$, where $B_1$ is the unit ball in $\mathbb R^3$. Both measures have finite second moment. Brenier's theorem \cite{Brenier1991} gives a convex potential $\Phi$ whose gradient is the transport map almost everywhere on the source measure.

To impose the gradient bounds on all of $\mathbb R^4$, we restrict the slopes of $\Phi$ to $\Omega_\ell$ by convex conjugation. Namely, define
\[
\varphi(z)=\sup_{p\in\Omega_\ell}\bigl(z\cdot p-\Phi^*(p)\bigr),
\qquad z=(t,x)\in\mathbb R^4.
\]
The function $\varphi$ is finite: the conjugate $\Phi^*$ is finite at some slope in $\Omega_\ell$ and has an affine lower bound. It is convex, with
$|\partial_t\varphi|\leq\ell$ and $|\nabla_x\varphi|\leq1$ almost everywhere. Since $\nabla\Phi\in\Omega_\ell$ for source-almost every $z$, Fenchel's equality gives $\varphi(z)=\Phi(z)$ at these points. Moreover, $\nabla\Phi(z)\in\partial\varphi(z)$. The source measure $w_N\,dz/S_N$ is absolutely continuous with respect to Lebesgue measure, and a finite convex function is differentiable almost everywhere. Hence $\varphi$ is differentiable at source-almost every point and $\nabla\varphi=\nabla\Phi$ there. In particular, $\nabla\varphi$ still transports $w_N/S_N$ to the uniform probability measure on $\Omega_\ell$.

By the Monge--Amp\`ere change-of-variables formula for the convex transport potential \cite[Theorem~3.1]{DePhilippisFigalli2014},
\begin{equation}\label{eq:MA-potential}
\det D_a^2\varphi(t,x)=\frac{|\Omega_\ell|}{S_N}w_N(t,x)
\quad\text{for almost every }(t,x)\in\{w_N>0\},
\end{equation}
where $D_a^2\varphi$ is the Alexandrov Hessian.

Let $\varphi_\varepsilon$ be the convolution of $\varphi$ with a nonnegative smooth compactly supported approximate identity. These functions are smooth and convex, have bounded Hessians for each $\varepsilon>0$, and preserve the two gradient bounds. At almost every Alexandrov point of $\varphi$,
$D^2\varphi_\varepsilon\to D_a^2\varphi$. Since both $T_g$ and $D^2\varphi_\varepsilon$ are nonnegative symmetric matrices, the trace integrands are nonnegative. Fatou's lemma applied to \eqref{eq:tensor-hessian-budget} therefore gives
\[
\int_I\int_{\mathbb R^3}\operatorname{tr}(T_gD_a^2\varphi)\,dx\,dt
\leq\ell\mathcal A_0+\mathcal A_1.
\]
The charge Hessian term need not be passed to the limit, because it was already discarded by positivity.

Using the matrix inequality, \eqref{eq:tensor-det}, and \eqref{eq:MA-potential}, we get
\begin{align*}
\ell\mathcal A_0+\mathcal A_1
&\geq4\int_{\{w_N>0\}}(\det T_g\det D_a^2\varphi)^{1/4}\,dt\,dx\\
&\geq c\|g\|_\infty^{-1/2}\left(\frac{\ell}{S_N}\right)^{1/4}
\int_{\{w_N>0\}}\rho_g^{3/2}w_N^{1/4}\,dt\,dx.
\end{align*}
Since $w_N\leq\rho_g^2$, the last integrand is at least $w_N$. Consequently,
\[
\ell\mathcal A_0+\mathcal A_1
\geq c\|g\|_\infty^{-1/2}\ell^{1/4}S_N^{3/4}.
\]
As $N\to\infty$, monotone convergence gives
$S_N\to\int_I\int_{\mathbb R^3}\rho_g^2\,dx\,dt$, which yields the estimate with free scale $\ell$.

If $\mathcal A_0,\mathcal A_1>0$ and the force budget is independent of $\ell$, choose
$\ell=\mathcal A_1/(3\mathcal A_0)$. Then \eqref{eq:convex-density-scale} yields
\[
\|\rho_g\|_{L^2}^{3/2}
\leq C\|g\|_\infty^{1/2}\mathcal A_0^{1/4}\mathcal A_1^{3/4}.
\]
Raising both sides to the power $4/3$ yields the optimized estimate. If either budget vanishes, the same estimate follows by letting $\ell\to\infty$ or $\ell\downarrow0$, as appropriate. Both estimates therefore hold without a positivity assumption on the budgets.
\end{proof}

The two forms of Proposition~\ref{prop:convex-density} serve different purposes. The free scale $\ell$ is useful for distributions localized to an energy layer. The optimized form is convenient for groups selected by fixed particle labels. In Section~3 we first apply the proposition to the full plasma--charge system, where the repulsive interactions make the signed force nonnegative.

\section{Global density and cumulative singular interactions}
\label{sec:global-density-interactions}

\subsection{A global density estimate}
\label{subsec:global-density}

We begin with a global space--time estimate for the plasma density. The main observation is that both Coulomb interactions have a favorable sign when the plasma and the point charge are tested together. The convex transport estimate from Section~\ref{sec:preliminaries} can therefore be applied without estimating the singular force by its absolute value.

\begin{proposition}\label{prop:global-density}
Let $(f,\xi,\eta)$ be a global classical solution with smooth, compactly supported initial data separated from the point charge. Then
\begin{equation}\label{eq:global-density}
\int_0^\infty\bigl(\|\rho(t)\|_{L^2(\mathbb R^3)}^2+\|E(t)\|_{L^6(\mathbb R^3)}^2\bigr)\,dt
\leq C(f_0,\mathcal E_0).
\end{equation}
\end{proposition}

\begin{proof}
Fix $T>0$ and apply Proposition~\ref{prop:convex-density} on $I=[0,T]$ with $g=f$, $b=E+F$, $s=0$, and $\theta=1$. The acceleration of the distinguished trajectory is $b_c(t)=E(t,\xi(t))$. The solution is classical, and on every bounded interval the transported support remains separated from the charge. Thus the integrability and trace assumptions of the proposition hold.

Let $\varphi(t,x)$ be a smooth convex test as in Proposition~\ref{prop:convex-density}. Its signed force term is
\begin{equation}\label{eq:full-force}
\mathcal F_\varphi=\int_I\int_{\mathbb R^3}\rho(E+F)\cdot\nabla_x\varphi\,dx\,dt
+\int_I E(t,\xi(t))\cdot\nabla_x\varphi(t,\xi(t))\,dt.
\end{equation}
We show that $\mathcal F_\varphi\geq0$.

First, the interaction among plasma particles can be symmetrized:
\begin{align*}
\int_{\mathbb R^3}\rho(t,x)E(t,x)\cdot\nabla_x\varphi(t,x)\,dx
=\frac12\iint_{\mathbb R^3\times\mathbb R^3}
\frac{\rho(t,x)\rho(t,y)(x-y)\cdot\bigl(\nabla_x\varphi(t,x)-\nabla_x\varphi(t,y)\bigr)}{|x-y|^3}\,dx\,dy.
\end{align*}
The right-hand side is nonnegative because the gradient of a convex function is monotone.

The Coulomb kernel gives $E(t,\xi(t))=-\int_{\mathbb R^3}\rho F\,dx$. Hence
\begin{align*}
&\int_{\mathbb R^3}\rho F\cdot\nabla_x\varphi\,dx
  +E(t,\xi(t))\cdot\nabla_x\varphi(t,\xi(t))\\
&\qquad=\int_{\mathbb R^3}\rho(t,x)\frac{(x-\xi(t))\cdot\bigl(\nabla_x\varphi(t,x)-\nabla_x\varphi(t,\xi(t))\bigr)}{|x-\xi(t)|^3}\,dx\geq0.
\end{align*}
This is the same monotonicity property, now applied to the plasma and the moving charge. Both terms in \eqref{eq:full-force} are therefore nonnegative. In Proposition~\ref{prop:convex-density} we may take $\mathcal B=0$ for every scale $\ell>0$.

For these choices, the endpoint budgets in Proposition~\ref{prop:convex-density} satisfy
$\mathcal A_0=2\|f_0\|_{L^1(\mathbb R^6)}+2$ and
\[
\mathcal A_1=\iint_{\mathbb R^6}|v|\bigl(f_0(x,v)+f(T,x,v)\bigr)\,dx\,dv+|\eta(0)|+|\eta(T)|.
\]
By Cauchy--Schwarz and Lemma~\ref{lem:basic-energy},
\[
\iint_{\mathbb R^6}|v|f(t,x,v)\,dx\,dv
\leq\|f_0\|_{L^1}^{1/2}
\left(\iint_{\mathbb R^6}|v|^2f(t,x,v)\,dx\,dv\right)^{1/2}
\leq C(f_0,\mathcal E_0).
\]
The point-charge velocity obeys the same uniform energy bound. Thus both budgets are bounded independently of $T$.

Since the signed force budget is zero for every $\ell$, the optimized estimate of Proposition~\ref{prop:convex-density} gives
\[
\int_0^T\int_{\mathbb R^3}\rho(t,x)^2\,dx\,dt
\leq C\|f_0\|_{L^\infty}^{2/3}\mathcal A_0^{1/3}\mathcal A_1
\leq C(f_0,\mathcal E_0).
\]
The Hardy--Littlewood--Sobolev inequality yields $\|E(t)\|_{L^6}\leq C\|\rho(t)\|_{L^2}$. The resulting bound for the sum of the two space--time integrals is independent of $T$. We may let $T$ tend to infinity, with no dependence on the separation distance.
\end{proof}

\subsection{The electric field along the charge}
\label{subsec:charge-field}

The estimate in Proposition~\ref{prop:global-density} controls the electric field in space and time, but not its value at the moving point $x=\xi(t)$. We obtain the missing control by separating plasma trajectories according to their velocities relative to the charge. Particles with moderate relative velocities have bounded spatial density. Those reaching large relative velocities carry only a small amount of mass, which can be estimated from the conserved energy.

Fix a bounded interval $I=[t_0,t_1]$ and write $\delta=|I|$. Set
\begin{equation}\label{eq:BQ-def}
B_I=\int_I\|\rho(t)\|_{L^2(\mathbb R^3)}^2\,dt,
\qquad Q_I=\int_I|E(t,\xi(t))|\,dt.
\end{equation}
By Proposition~\ref{prop:global-density},
\begin{equation}\label{eq:B-star}
B_*:=\int_0^\infty\|\rho(t)\|_{L^2(\mathbb R^3)}^2\,dt\leq C(f_0,\mathcal E_0).
\end{equation}

For $z=(x,v)\in\mathbb R^6$, let $(X_z(t),V_z(t))$ be the characteristic with $(X_z(t_0),V_z(t_0))=z$, and define
\[
d\mu(z)=f(t_0,z)\,dz,\qquad G_I(z)=\int_I|E(t,X_z(t))|\,dt.
\]
Since the flow preserves volume and the distribution is constant along characteristics, for every nonnegative measurable $\Phi$,
\begin{equation}\label{eq:label-flow}
\int_{\mathbb R^6}\Phi(X_z(t),V_z(t))\,d\mu(z)
=\iint_{\mathbb R^6}\Phi(x,v)f(t,x,v)\,dx\,dv.
\end{equation}

\begin{lemma}\label{lem:accumulated-plasma-field}
For every bounded interval $I=[t_0,t_1]$,
\begin{equation}\label{eq:accumulated-field-square}
\int_{\mathbb R^6}G_I(z)^2\,d\mu(z)\leq C(f_0,\mathcal E_0)\,\delta B_I
\end{equation}
and
\begin{equation}\label{eq:accumulated-field-first}
\int_{\mathbb R^6}G_I(z)\,d\mu(z)
\leq C(f_0,\mathcal E_0)\,\delta^{1/2}B_I^{1/2}.
\end{equation}
\end{lemma}

\begin{proof}
Cauchy--Schwarz in time gives $G_I(z)^2\leq\delta\int_I|E(t,X_z(t))|^2\,dt$. Integrating in $z$ and using \eqref{eq:label-flow},
\[
\int_{\mathbb R^6}G_I^2\,d\mu
\leq\delta\int_I\int_{\mathbb R^3}\rho|E|^2\,dx\,dt.
\]
H\"older's inequality and Lemma~\ref{lem:basic-energy} imply
\[
\int_{\mathbb R^3}\rho|E|^2\,dx
\leq\|\rho\|_{L^{3/2}}\|E\|_{L^6}^2
\leq C(f_0,\mathcal E_0)\|\rho\|_{L^2}^2,
\]
where we also used the Hardy--Littlewood--Sobolev inequality. Integration over $I$ proves the first assertion.

For the first moment, \eqref{eq:label-flow} gives
\[
\int_{\mathbb R^6}G_I\,d\mu
=\int_I\int_{\mathbb R^3}\rho|E|\,dx\,dt
\leq C(f_0,\mathcal E_0)\int_I\|\rho(t)\|_{L^2}\,dt.
\]
Here we used the uniform bound for $\|\rho(t)\|_{L^{6/5}}$ and the $L^{6/5}$--$L^6$ H\"older inequality. Cauchy--Schwarz in time gives the second assertion. 
\end{proof}

The field at the charge requires more than the first-moment estimate. In the next proposition, the second-moment bound controls the mass of particles that acquire a large relative velocity. The radial inequality \eqref{eq:radial-trajectory-bound} then controls the Coulomb interaction of this exceptional group.

\begin{proposition}\label{prop:local-charge-field}
There exist constants $c_0,L_0>0$, depending only on $\|f_0\|_{L^1(\mathbb R^6)}$, $\|f_0\|_{L^\infty(\mathbb R^6)}$, and $\mathcal E_0$, such that every bounded interval $I=[t_0,t_1]$ satisfies
\begin{equation}\label{eq:local-Q-condition}
\delta^2B_I\leq c_0\quad\Longrightarrow\quad Q_I\leq L_0.
\end{equation}
\end{proposition}

\begin{proof}
Fix $\Lambda\geq8$ and put $R=\Lambda(1+Q_I)$. Divide the particle labels at time $t_0$ into the two sets
\[
\mathcal L_I=\bigl\{z:\sup_{t\in I}|V_z(t)-\eta(t)|\leq R\bigr\},\qquad
\mathcal H_I=\mathbb R^6\setminus\mathcal L_I.
\]
These sets contain the particles whose relative velocities stay below $R$, and those whose relative velocities exceed $R$ at least once, respectively.

Let $\rho_{\mathcal L}(t,x)$ be the spatial density of the particles with labels in $\mathcal L_I$. The flow representation gives $0\leq\rho_{\mathcal L}\leq\rho$. Moreover, all these particles have velocities in $\{|v-\eta(t)|\leq R\}$ at time $t$. Conservation of the phase-space $L^\infty$ bound therefore implies
$\|\rho_{\mathcal L}(t)\|_{L^\infty}\leq C\|f_0\|_{L^\infty}R^3$. The Coulomb interpolation inequality \eqref{eq:coulomb-pointwise}, with $p=2$, and H\"older's inequality in time give
\begin{align}\label{eq:low-charge-field}
\int_I\int_{\mathbb R^3}\frac{\rho_{\mathcal L}(t,x)}{|x-\xi(t)|^2}\,dx\,dt
&\leq C(f_0)R\int_I\|\rho(t)\|_{L^2}^{2/3}\,dt\notag\\
&\leq C(f_0)R\,\delta^{2/3}B_I^{1/3}.
\end{align}
In particular, the right-hand side is at most $C(f_0)\Lambda c_0^{1/3}(1+Q_I)$ if $\delta^2B_I\leq c_0$.

The mass $m:=\mu(\mathcal H_I)$ of the high-velocity group is controlled by the relative energy estimate \eqref{eq:relative-weight} gives
\begin{equation}\label{eq:relative-energy-integrated-3}
a(t,X_z(t),V_z(t))\leq a(t_0,z)+G_I(z)+Q_I,\qquad t\in I.
\end{equation}
For $z\in\mathcal H_I$, the left-hand side exceeds $R$ at some time. Since $Q_I\leq R/\Lambda$, the two conditions $a(t_0,z)<R/4$ and $G_I(z)<R/4$ would make the right-hand side smaller than $R$ at every time. Thus
\[
\mathcal H_I\subset\{z:a(t_0,z)\geq R/4\}\cup\{z:G_I(z)\geq R/4\}.
\]
Chebyshev's inequality, Lemma~\ref{lem:basic-energy}, and \eqref{eq:accumulated-field-square} now give
\begin{equation}\label{eq:high-group-mass}
m\leq\frac{C(f_0,\mathcal E_0)}{R^2}(1+\delta B_I).
\end{equation}
This is the only place where a high-velocity mass estimate is needed.

We use the radial inequality \eqref{eq:radial-trajectory-bound} on the same labels. Its integration over $\mathcal H_I$ gives
\begin{align*}
\int_{\mathcal H_I}\int_I\frac{dt}{|X_z(t)-\xi(t)|^2}\,d\mu(z)
\leq{}&\sum_{i=0}^1\int_{\mathcal H_I}|V_z(t_i)-\eta(t_i)|\,d\mu(z)\\
&+\int_{\mathcal H_I}G_I(z)\,d\mu(z)+mQ_I.
\end{align*}
For each endpoint, Cauchy--Schwarz and Lemma~\ref{lem:basic-energy} bound the integral by $C(f_0,\mathcal E_0)m^{1/2}$. The field term is bounded by
\[
\int_{\mathcal H_I}G_I\,d\mu
\leq m^{1/2}\left(\int_{\mathbb R^6}G_I^2\,d\mu\right)^{1/2}
\leq C(f_0,\mathcal E_0)m^{1/2}(\delta B_I)^{1/2}.
\]
It follows that
\begin{equation}\label{eq:high-charge-field-before}
\int_{\mathcal H_I}\int_I\frac{dt}{|X_z(t)-\xi(t)|^2}\,d\mu(z)
\leq C(f_0,\mathcal E_0)m^{1/2}\bigl(1+(\delta B_I)^{1/2}\bigr)+mQ_I.
\end{equation}

Assume now that $\delta^2B_I\leq c_0\leq1$. Since $B_I\leq B_*$,
\[
\delta B_I=(\delta^2B_I)^{1/2}B_I^{1/2}
\leq c_0^{1/2}B_*^{1/2}\leq C(f_0,\mathcal E_0).
\]
By \eqref{eq:high-group-mass}, $m\leq C(f_0,\mathcal E_0)R^{-2}$. Thus \eqref{eq:high-charge-field-before} and the choice of $R$ imply
\begin{equation}\label{eq:high-charge-field}
\int_{\mathcal H_I}\int_I\frac{dt}{|X_z(t)-\xi(t)|^2}\,d\mu(z)
\leq C(f_0,\mathcal E_0)\bigl(\Lambda^{-1}+\Lambda^{-2}\bigr).
\end{equation}
In this estimate, $Q_I$ is multiplied by the mass of the exceptional group, not by the mass of the full plasma.

The absolute value of the Coulomb field at the charge is bounded by the inverse-square interaction of all the particles. By the flow representation, the latter is the sum of the two contributions estimated in \eqref{eq:low-charge-field} and \eqref{eq:high-charge-field}. Hence
\[
Q_I\leq C(f_0,\mathcal E_0)
\bigl[\Lambda c_0^{1/3}(1+Q_I)+\Lambda^{-1}+\Lambda^{-2}\bigr].
\]
Fix $\Lambda\geq8$ and then take $c_0>0$ sufficiently small that the coefficient of $Q_I$ on the right is at most $1/2$. This term can be absorbed, leaving a constant $L_0$ that depends only on the initial finite-energy data. 
\end{proof}

\subsection{Sublinear growth of singular interactions}
\label{subsec:sublinear-interactions}

The local estimate in Proposition~\ref{prop:local-charge-field} applies whenever $|I|^2B_I$ is small. Choose the time intervals according to the density they contain. The global finiteness of $B_*$ allows us to control how many such intervals are needed, and the radial inequality then gives the same growth rate for the inverse-square interaction.

\begin{theorem}\label{thm:sublinear-interactions}
Let $(f,\xi,\eta)$ be a global classical solution under the assumptions of Section~\ref{sec:preliminaries}. For every $T\geq0$,
\begin{equation}\label{eq:sublinear-interactions}
\int_0^T|E(t,\xi(t))|\,dt
+\int_0^T\int_{\mathbb R^3}\frac{\rho(t,x)}{|x-\xi(t)|^2}\,dx\,dt
\leq C(f_0,\mathcal E_0)(1+T)^{2/3}.
\end{equation}
\end{theorem}

\begin{proof}
To estimate $Q_{[0,T]}$, fix $T>0$ and let $c_0,L_0$ be the constants in Proposition~\ref{prop:local-charge-field}. For a fixed left endpoint $s$, define
\[\Phi(s,t)=(t-s)^2\int_s^t\|\rho(\tau)\|_{L^2}^2\,d\tau,\qquad t\ge s.\]
The function $t\mapsto\Phi(s,t)$ is continuous and nondecreasing. Starting at $t_0=0$, choose $t_j$ as the first time for which $\Phi(t_{j-1},t_j)=c_0$, unless $T$ is reached first. Each complete interval $I_j=[t_{j-1},t_j]$ satisfies $|I_j|^2B_{I_j}=c_0$, while the final interval satisfies $|I_j|^2B_{I_j}\leq c_0$.

If $B_*=0$, the density vanishes, so the conclusion is immediate. Suppose $B_*>0$. Every complete interval has length at least $(c_0/B_*)^{1/2}$, and the partition is finite. Let $N$ be the number of complete intervals. H\"older's inequality gives
\begin{align*}
Nc_0^{1/3}
&=\sum_{j=1}^{N}|I_j|^{2/3}B_{I_j}^{1/3}\\
&\leq\left(\sum_{j=1}^{N}|I_j|\right)^{2/3}
\left(\sum_{j=1}^{N}B_{I_j}\right)^{1/3}
\leq T^{2/3}B_*^{1/3}.
\end{align*}
In view of \eqref{eq:B-star}, $N\leq C(f_0,\mathcal E_0)T^{2/3}$. Proposition~\ref{prop:local-charge-field} applies to every interval in the partition. Summing its bounds, including the possible final interval, yields
\begin{equation}\label{eq:sublinear-Q}
Q_{[0,T]}\leq L_0(N+1)\leq C(f_0,\mathcal E_0)(1+T)^{2/3}.
\end{equation}

It remains to estimate the inverse-square interaction. Label characteristics at time zero and put $d\mu(z)=f_0(z)\,dz$. The radial inequality \eqref{eq:radial-trajectory-bound} gives, after integration in $z$,
\begin{align*}
\int_0^T\int_{\mathbb R^3}\frac{\rho(t,x)}{|x-\xi(t)|^2}\,dx\,dt
\leq{}&\sum_{t\in\{0,T\}}\iint_{\mathbb R^6}|v-\eta(t)|f(t,x,v)\,dx\,dv\\
&+\int_{\mathbb R^6}G_{[0,T]}(z)\,d\mu(z)
+\|f_0\|_{L^1}Q_{[0,T]}.
\end{align*}
The endpoint integrals are bounded by $C(f_0,\mathcal E_0)$ using Lemma~\ref{lem:basic-energy}. By Lemma~\ref{lem:accumulated-plasma-field} and \eqref{eq:B-star},
\[
\int_{\mathbb R^6}G_{[0,T]}\,d\mu
\leq C(f_0,\mathcal E_0)T^{1/2}B_{[0,T]}^{1/2}
\leq C(f_0,\mathcal E_0)T^{1/2}.
\]
Combining these bounds with \eqref{eq:sublinear-Q} controls the inverse-square interaction by $C(f_0,\mathcal E_0)(1+T^{1/2}+T^{2/3})$. Since $T^{1/2}\leq(1+T)^{2/3}$, both cumulative interactions obey the asserted sublinear bound.
\end{proof}

\section{Polynomial growth of energy moments}
\label{sec:moment-growth}

\subsection{Large accumulated fields}
\label{subsec:large-fields}

For higher relative energy moments, we also need to control the field along plasma characteristics. Theorem~\ref{thm:sublinear-interactions} controls the field along the point charge. We estimate the mass of particles that experience a large accumulated field; this gives the fourth-moment bound.

Fix $T>0$ and write $S=1+T$. We label characteristics at time zero by $z=(x,v)$ and set
\[
d\mu(z)=f_0(z)\,dz,\qquad a_z(t)=a(t,X_z(t),V_z(t)),\qquad a_0(z)=a_z(0),
\]
where $a(t,x,v)=(1+|v-\eta(t)|^2+2/|x-\xi(t)|)^{1/2}$. Define
\[
G_T(z)=\int_0^T|E(t,X_z(t))|\,dt,\qquad Q_T=\int_0^T|E(t,\xi(t))|\,dt.
\]
The relative energy estimate \eqref{eq:relative-weight} gives
\begin{equation}\label{eq:particle-energy-growth}
a_z(t)\leq a_0(z)+G_T(z)+Q_T,\qquad 0\leq t\leq T.
\end{equation}
By Theorem~\ref{thm:sublinear-interactions}, $Q_T\leq C(f_0,\mathcal E_0)S^{2/3}$. It remains to understand $G_T$.

\begin{lemma}\label{lem:large-accumulated-tail}
For $\lambda\geq\max\{1,Q_T\}$, let
\[
A_\lambda=\{z:G_T(z)>\lambda,\ G_T(z)\geq a_0(z)\}.
\]
Then
\begin{equation}\label{eq:large-field-tail}
\mu(A_\lambda)\leq C(f_0,\mathcal E_0)S^3\lambda^{-15/2}.
\end{equation}
\end{lemma}

\begin{proof}
Fix $\lambda$ and put $m=\mu(A_\lambda)$ and $J=\int_{A_\lambda}G_T\,d\mu$. We may assume $m>0$. By the definition of $A_\lambda$,
\begin{equation}\label{eq:group-J-lower}
J\geq\lambda m\geq m.
\end{equation}
For $z\in A_\lambda$, \eqref{eq:particle-energy-growth} also gives
\[
a_z(t)\leq a_0(z)+G_T(z)+Q_T\leq3G_T(z),\qquad 0\leq t\leq T.
\]

Let $f_A$ be the distribution transported from $f_0\mathbf1_{A_\lambda}$ and write $\rho_A=\int_{\mathbb R^3}f_A\,dv$. The set $A_\lambda$ is a measurable set of initial labels, fixed once $T$ and $\lambda$ are chosen. Transporting $f_0\mathbf1_{A_\lambda}$ by the characteristic flow therefore gives a solution of the homogeneous transport equation
\[
\partial_t f_A+v\cdot\nabla_x f_A+\operatorname{div}_v((E+F)f_A)=0
\]
in distributions. The flow preserves phase-space volume, so $0\leq f_A\leq f$, $\|f_A\|_\infty\leq\|f_0\|_\infty$, and $\int\rho_A(t,x)\,dx=m$ for every $t$. Since the classical solution has compact support on $[0,T]$ and its flow is continuous there, the transported density has strong time traces in $L^1(\mathbb R^6;(1+|v|)\,dx\,dv)$. These are the endpoint traces required by Proposition~\ref{prop:convex-density}.

We check the budgets in that proposition. On $A_\lambda$, $|V_z(t)|\leq a_z(t)+|\eta(t)|$. The uniform bound on $\eta$ from Lemma~\ref{lem:basic-energy} and \eqref{eq:group-J-lower} therefore give
\[
\iint_{\mathbb R^6}|v|f_A(t,x,v)\,dx\,dv\leq C(f_0,\mathcal E_0)J,\qquad t=0,T.
\]
The self-consistent force has the exact integral
\[
\int_0^T\int_{\mathbb R^3}|E|\rho_A\,dx\,dt=J.
\]
For the force of the point charge, integrate the radial bound \eqref{eq:radial-trajectory-bound} over $A_\lambda$. The endpoint relative velocities cost at most $CJ$, and the last term is $mQ_T\leq m\lambda\leq J$. Hence
\[
\int_0^T\int_{\mathbb R^3}|F|\rho_A\,dx\,dt\leq C(f_0,\mathcal E_0)J.
\]

Apply the optimized estimate \eqref{eq:convex-density-optimized} with $g=f_A$, $s=0$, and $\theta=0$. The mass budget is $\mathcal A_0=2m$. The two force integrals give a lower bound for the signed force, independent of the scale $\ell$; together with the endpoint estimates they give $\mathcal A_1\leq CJ$. We obtain
\begin{equation}\label{eq:fixed-label-density}
\|\rho_A\|_{L^2((0,T)\times\mathbb R^3)}^2\leq C m^{1/3}J.
\end{equation}

To estimate $J$, interpolate at each time:
\[
\|\rho_A(t)\|_{L^{6/5}}\leq\|\rho_A(t)\|_{L^1}^{2/3}\|\rho_A(t)\|_{L^2}^{1/3}=m^{2/3}\|\rho_A(t)\|_{L^2}^{1/3}.
\]
By H\"older's inequality in space and then in time, and by Proposition~\ref{prop:global-density},
\begin{align*}
J&\leq\int_0^T\|E(t)\|_{L^6}\|\rho_A(t)\|_{L^{6/5}}\,dt\\
&\leq C S^{1/3}m^{2/3}\|\rho_A\|_{L^2((0,T)\times\mathbb R^3)}^{1/3}
\leq C S^{1/3}m^{13/18}J^{1/6}.
\end{align*}
Since $J\geq\lambda m>0$, dividing by $J^{1/6}$ gives
\[
\lambda^{5/6}m^{1/9}\leq C S^{1/3}.
\]
Taking the ninth power gives the asserted tail estimate. 
\end{proof}

The tail estimate allows us to propagate moments below the fourth order, including the fourth moment itself. Only the part of the plasma whose present energy is much larger than its initial energy needs the tail bound.

\begin{proposition}\label{prop:low-moments}
Let $2<k\leq4$ and assume $H_k(0)<\infty$. Then, for every $T\geq0$,
\begin{equation}\label{eq:low-moment-growth}
\sup_{0\leq t\leq T}\bigl(1+H_k(t)\bigr)
\leq C_k(f_0,\mathcal E_0,H_k(0))(1+T)^{\frac23(k-2)}.
\end{equation}
In particular, $H_4(0)<\infty$ implies $\sup_{0\leq t\leq T}H_4(t)\leq C(f_0,\mathcal E_0,H_4(0))(1+T)^{4/3}$.
\end{proposition}

\begin{proof}
Fix $T>0$ and set $S=1+T$. By Theorem~\ref{thm:sublinear-interactions}, we can choose
\[
R_0=C_0S^{2/3}\geq4(1+Q_T),
\]
with $C_0$ depending only on the finite-energy bounds. Fix $t\in[0,T]$ and divide the particles according to the relation between $a_z(t)$ and $a_0(z)$.

On $\{a_z(t)\leq4a_0(z)\}$, their contribution is at most $4^kH_k(0)$. On $\{4a_0(z)<a_z(t)\leq R_0\}$, we use the uniformly bounded second moment:
\[
\int_{\{4a_0<a_z(t)\leq R_0\}}a_z(t)^k\,d\mu
\leq R_0^{k-2}H_2(t)\leq CS^{\frac23(k-2)}.
\]

It remains to treat
\[
\Gamma_t=\{z:a_z(t)>R_0,\ a_z(t)>4a_0(z)\}.
\]
For $z\in\Gamma_t$, \eqref{eq:particle-energy-growth}, $a_0(z)<a_z(t)/4$, and $Q_T\leq R_0/4<a_z(t)/4$ imply
\[
G_T(z)\geq\tfrac12a_z(t)\geq a_0(z).
\]
Thus, for every $r\geq R_0$, the set $\{z\in\Gamma_t:a_z(t)>r\}$ is contained in $A_{r/2}$. We may apply Lemma~\ref{lem:large-accumulated-tail}, since $r/2\geq\max\{1,Q_T\}$, and obtain
\[
\mu\bigl(\{z\in\Gamma_t:a_z(t)>r\}\bigr)\leq CS^3r^{-15/2}.
\]
The layer-cake formula, with its contribution at the lower endpoint included, gives
\begin{align*}
\int_{\Gamma_t}a_z(t)^k\,d\mu
&=R_0^k\mu(\Gamma_t)+k\int_{R_0}^\infty r^{k-1}\mu\bigl(\{z\in\Gamma_t:a_z(t)>r\}\bigr)\,dr\\
&\leq C_kS^3R_0^{k-15/2}\leq C_kS^{\frac23(k-3)}.
\end{align*}
Here $k\leq4<15/2$ ensures convergence of the integral. The last contribution is bounded by $C_kS^{\frac23(k-2)}$ because $S\geq1$. Adding the three groups and taking the supremum over $t\in[0,T]$ finishes the estimate.
\end{proof}

\subsection{Dynamic energy layers}
\label{subsec:dynamic-layers}

For moments above order four, fixed initial labels are no longer sufficient. We use dyadic layers of the current relative energy. Since particles can enter and leave a layer, its transport equation has a source term. The next estimate controls this source and the point-charge force.

For the rest of this section, assume $H_4(0)<\infty$ and write $S=1+T$. Choose $\chi\in C_c^\infty((1/2,4))$ with $0\leq\chi\leq1$ and $\chi=1$ on $[1,2]$. For every dyadic $R=2^j$, $j\geq0$, set
\[
g_R=f\chi(a/R)^2,\qquad \rho_R(t,x)=\int_{\mathbb R^3}g_R(t,x,v)\,dv,
\]
and define the larger layer and its maximal mass by
\[
\sigma_R(t,x)=\int_{\{R/2\leq a(t,x,v)\leq4R\}}f(t,x,v)\,dv,\qquad
m_R=\sup_{0\leq t\leq T}\int_{\mathbb R^3}\sigma_R(t,x)\,dx.
\]
Then $0\leq\rho_R\leq\sigma_R\leq\rho$. The larger density is useful because the derivative of the cutoff is supported on the enlarged layer.

\begin{lemma}\label{lem:energy-layer-budgets}
For every dyadic $R\geq1$,
\begin{equation}\label{eq:layer-mass-infty}
m_R\leq C\min\{R^{-2},S^{4/3}R^{-4}\},\qquad
\|\rho_R(t)\|_{L^\infty(\mathbb R^3)}\leq CR^3.
\end{equation}
Let $\mathcal L=\partial_t+v\cdot\nabla_x+(E+F)\cdot\nabla_v$. The truncated distribution satisfies $\mathcal Lg_R=s_R$, where
\begin{equation}\label{eq:layer-source-pointwise}
\int_{\mathbb R^3}|s_R(t,x,v)|\,dv
\leq\frac C R\bigl(|E(t,x)|+|E(t,\xi(t))|\bigr)\sqrt{\rho_R(t,x)\sigma_R(t,x)}.
\end{equation}
Put
\[
J_R=\int_0^T\int_{\mathbb R^3}|E(t,x)|\sqrt{\rho_R(t,x)\sigma_R(t,x)}\,dx\,dt.
\]
Then
\begin{equation}\label{eq:layer-source-integral}
\int_0^T\iint_{\mathbb R^6}(R+|v|)|s_R|\,dx\,dv\,dt
\leq C(J_R+Q_Tm_R)
\end{equation}
and
\begin{equation}\label{eq:layer-singular-force}
\int_0^T\int_{\mathbb R^3}\frac{\rho_R(t,x)}{|x-\xi(t)|^2}\,dx\,dt
\leq C\bigl((R+Q_T)m_R+J_R\bigr).
\end{equation}
\end{lemma}

\begin{proof}
On the support of $\sigma_R$, $a\geq R/2$. Lemma~\ref{lem:basic-energy} and Proposition~\ref{prop:low-moments} therefore give
\[
\int_{\mathbb R^3}\sigma_R(t,x)\,dx
\leq\min\{4R^{-2}H_2(t),\,16R^{-4}H_4(t)\}
\leq C\min\{R^{-2},S^{4/3}R^{-4}\}.
\]
Also, $a\leq4R$ on this enlarged layer, so
\[
|v-\eta(t)|\leq4R,\qquad |x-\xi(t)|\geq\frac1{8R^2}.
\]
In particular $|v|\leq CR$, because $R\geq1$ and $\eta$ is uniformly bounded. Conservation of the phase-space $L^\infty$ norm gives $\|\rho_R(t)\|_\infty\leq C\|f_0\|_\infty R^3$.

Since $\mathcal Lf=0$, differentiation of the square cutoff gives
\[
s_R=\frac2R f\chi(a/R)\chi'(a/R)\mathcal La.
\]
The relative energy identity \eqref{eq:relative-weight} implies $|\mathcal La|\leq|E(t,x)|+|E(t,\xi(t))|$. By Cauchy--Schwarz in velocity and the support of $\chi'$,
\begin{align*}
\int_{\mathbb R^3}f|\chi(a/R)\chi'(a/R)|\,dv
&\leq\left(\int f\chi(a/R)^2\,dv\right)^{1/2}
\left(\int f|\chi'(a/R)|^2\,dv\right)^{1/2}\\
&\leq C\sqrt{\rho_R\sigma_R}.
\end{align*}
This gives \eqref{eq:layer-source-pointwise}. Since $|v|\leq CR$ on the support of $s_R$, multiplying the pointwise bound by $R+|v|$ and integrating gives
\[
\int_0^T\iint(R+|v|)|s_R|\,dx\,dv\,dt
\leq C J_R+C\int_0^T|E(t,\xi(t))|\int_{\mathbb R^3}\sqrt{\rho_R\sigma_R}\,dx\,dt.
\]
Cauchy--Schwarz in space and $\rho_R\leq\sigma_R$ show that $\int\sqrt{\rho_R\sigma_R}\,dx\leq m_R$. Hence the last term is at most $CQ_Tm_R$.

For the singular force, use the radial velocity
\[
p(t,x,v)=(v-\eta(t))\cdot\frac{x-\xi(t)}{|x-\xi(t)|}.
\]
The radial identity \eqref{eq:radial-identity} gives
\[
\mathcal Lp=\frac{|v-\eta|^2-p^2}{|x-\xi|}+\frac1{|x-\xi|^2}
+(E(t,x)-E(t,\xi(t)))\cdot\frac{x-\xi}{|x-\xi|}.
\]
Dropping the nonnegative first term,
\[
\frac1{|x-\xi(t)|^2}\leq\mathcal Lp+|E(t,x)|+|E(t,\xi(t))|.
\]
Multiply by $g_R$ and integrate in time and phase space. The equation $\mathcal Lg_R=s_R$ and phase-space integration by parts give
\[
\int_0^T\iint g_R\mathcal Lp\,dx\,dv\,dt
=\left[\iint p g_R\,dx\,dv\right]_0^T-\int_0^T\iint p s_R\,dx\,dv\,dt.
\]
The calculation is justified by smooth cutoffs: $g_R$ and $s_R$ are supported where $|x-\xi|\geq(8R^2)^{-1}$, and the classical solution has compact support on $[0,T]$. On these supports $|p|\leq4R$. Thus the endpoint terms are bounded by $CRm_R$, and the source term by $C(J_R+Q_Tm_R)$.

Finally, $\rho_R\leq\sqrt{\rho_R\sigma_R}$ implies
\[
\int_0^T\int_{\mathbb R^3}|E|\rho_R\,dx\,dt\leq J_R,
\]
while $\int_0^T|E(t,\xi(t))|\int\rho_R(t,x)\,dx\,dt\leq Q_Tm_R$. These bounds control all the field terms and give \eqref{eq:layer-singular-force}. Both the source and the singular force are thus measured by the same quantity $J_R$ and the mass of the enlarged layer.
\end{proof}

\subsection{The accumulated field and higher moments}
\label{subsec:accumulated-field-higher}

Apply the scaled convex transport estimate to each moving energy layer. The source term in Lemma~\ref{lem:energy-layer-budgets} depends on the layer density through a small power of its $L^2$ norm, which permits the nonlinear absorption below.

\begin{proposition}\label{prop:global-accumulated-field}
If $H_4(0)<\infty$, then, for every $T\geq0$,
\begin{equation}\label{eq:global-accumulated-field}
\int_0^T\|E(t)\|_{L^\infty(\mathbb R^3)}\,dt
\leq C(f_0,\mathcal E_0,H_4(0))(1+T).
\end{equation}
\end{proposition}

\begin{proof}
Fix $T>0$, set $S=1+T$, and use the notation of Lemma~\ref{lem:energy-layer-budgets}. Write
\[
L_R=\|\rho_R\|_{L^2((0,T)\times\mathbb R^3)}.
\]
We apply the scaled convex transport estimate \eqref{eq:convex-density-scale} to $g_R$, with $b=E+F$, $s=s_R$, $\theta=0$, and $\ell=R$.

The masses of $g_R$ at the endpoints are at most $m_R$, and their first velocity moments are at most $CRm_R$. By \eqref{eq:layer-source-integral}, the source terms contribute at most $C(J_R+Q_Tm_R)$ to $R\mathcal A_0+\mathcal A_1$. The signed force can be bounded below, uniformly over the convex tests, by minus the absolute force integral. Lemma~\ref{lem:energy-layer-budgets} gives
\[
\int_0^T\int_{\mathbb R^3}(|E|+|F|)\rho_R\,dx\,dt
\leq C\bigl((R+Q_T)m_R+J_R\bigr).
\]
These estimates also verify the force, source and trace requirements in Proposition~\ref{prop:convex-density}. Since $\|g_R\|_\infty\leq\|f_0\|_\infty$, we obtain
\begin{equation}\label{eq:layer-convex-bound}
L_R^{3/2}\leq CR^{-1/4}\bigl((R+Q_T)m_R+J_R\bigr).
\end{equation}

Since $0\leq\sigma_R\leq\rho$, Proposition~\ref{prop:global-density} gives
\[
\|\sigma_R\|_{L^2((0,T)\times\mathbb R^3)}\leq\|\rho\|_{L^2((0,\infty)\times\mathbb R^3)}\leq C.
\]
The same proposition gives $\|E\|_{L^2(0,T;L^6)}\leq C$. Thus the bound for the enlarged layer is uniform in $R$ and $T$. H\"older's inequality in space and interpolation give
\begin{align*}
\|\sqrt{\rho_R\sigma_R}\|_{L^{6/5}}
&\leq\|\rho_R\|_{L^{6/5}}^{1/2}\|\sigma_R\|_{L^{6/5}}^{1/2}\\
&\leq m_R^{2/3}\|\rho_R\|_{L^2}^{1/6}\|\sigma_R\|_{L^2}^{1/6}.
\end{align*}
By H\"older's inequality in time, with exponents $2,12,12,3$,
\begin{align*}
J_R&\leq m_R^{2/3}\int_0^T\|E(t)\|_{L^6}
\|\rho_R(t)\|_{L^2}^{1/6}\|\sigma_R(t)\|_{L^2}^{1/6}\,dt\\
&\leq CS^{1/3}m_R^{2/3}L_R^{1/6}.
\end{align*}
Substituting this estimate into \eqref{eq:layer-convex-bound}, we have
\[
L_R^{3/2}\leq A_R+B_RL_R^{1/6},\qquad
A_R=CR^{-1/4}(R+Q_T)m_R,\quad
B_R=CS^{1/3}R^{-1/4}m_R^{2/3}.
\]
If $L_R^{3/2}\leq2A_R$, then $L_R\leq(2A_R)^{2/3}$. Otherwise $L_R^{3/2}\leq2B_RL_R^{1/6}$ and $L_R\leq(2B_R)^{3/4}$. Thus
\begin{equation}\label{eq:layer-nonlinear-absorption}
L_R\leq C\bigl(A_R^{2/3}+B_R^{3/4}\bigr).
\end{equation}

We estimate the field generated by this layer. Let
\[
W_R=\int_0^T\big\||\cdot|^{-2}*\rho_R(t)\big\|_{L^\infty(\mathbb R^3)}\,dt.
\]
The Coulomb interpolation estimate \eqref{eq:coulomb-pointwise} with $p=2$ and \eqref{eq:layer-mass-infty} give
\[
W_R\leq CR\int_0^T\|\rho_R(t)\|_{L^2}^{2/3}\,dt
\leq CS^{2/3}R L_R^{2/3}.
\]
Combining this inequality with \eqref{eq:layer-nonlinear-absorption}, we obtain the single-layer bound
\begin{equation}\label{eq:single-layer-field}
W_R\leq C\Bigl[S^{2/3}R^{8/9}(R+Q_T)^{4/9}m_R^{4/9}
+S^{5/6}R^{7/8}m_R^{1/3}\Bigr].
\end{equation}

It remains to sum over the dyadic layers. The two mass estimates in \eqref{eq:layer-mass-infty} apply at different scales. Put $R_*=S^{2/3}$. By Theorem~\ref{thm:sublinear-interactions}, $Q_T\leq CR_*$. For $1\leq R\leq R_*$, use $m_R\leq CR^{-2}$ and $R+Q_T\leq CR_*$. Then \eqref{eq:single-layer-field} gives
\begin{equation}\label{eq:low-layer-sum-bound}
W_R\leq C\left[S^{26/27}+S^{35/36}\left(\frac R{R_*}\right)^{5/24}\right].
\end{equation}
The first term is independent of $R$, and there are at most $C\log(2+S)$ low dyadic layers. The second term forms a geometric series.

For $R>R_*$, use $m_R\leq CS^{4/3}R^{-4}$ and $R+Q_T\leq CR$. Now
\begin{equation}\label{eq:high-layer-sum-bound}
W_R\leq C\left[S^{26/27}\left(\frac R{R_*}\right)^{-4/9}
+S^{35/36}\left(\frac R{R_*}\right)^{-11/24}\right].
\end{equation}
Both sums over the high dyadic layers converge geometrically. Combining the two ranges yields
\begin{equation}\label{eq:dyadic-field-sum}
\sum_{j\geq0}W_{2^j}
\leq C\bigl(S^{26/27}\log(2+S)+S^{35/36}\bigr)\leq CS.
\end{equation}

Since $a(t,x,v)\geq1$ and $\chi=1$ on $[1,2]$, every phase point belongs to a dyadic layer on which the cutoff equals one. Thus $f\leq\sum_{j\geq0}g_{2^j}$ and $\rho\leq\sum_{j\geq0}\rho_{2^j}$. The Coulomb kernel satisfies $|E|\leq|\cdot|^{-2}*\rho$, so Tonelli's theorem gives
\[
\int_0^T\|E(t)\|_\infty\,dt\leq\sum_{j\geq0}W_{2^j}\leq CS.
\]
\end{proof}

We can now turn the field estimate into moment bounds. For any particle trajectory, Proposition~\ref{prop:global-accumulated-field} and Theorem~\ref{thm:sublinear-interactions} control both field terms in \eqref{eq:particle-energy-growth}. The remaining argument uses the uniformly bounded second moment.

\begin{theorem}\label{thm:all-moment-growth}
Let $K>2$ and assume $H_K(0)<\infty$. For every global classical solution under the assumptions of Section~\ref{sec:preliminaries},
\begin{equation}\label{eq:all-moment-growth}
\sup_{0\leq t\leq T}\bigl(1+H_K(t)\bigr)
\leq C_K(f_0,\mathcal E_0,H_K(0))(1+T)^{K-2},\qquad T\geq0.
\end{equation}
\end{theorem}

\begin{proof}
For $2<K\leq4$, Proposition~\ref{prop:low-moments} gives the stronger exponent $\frac23(K-2)$. We therefore assume $K>4$. Since $a_0\geq1$, $H_K(0)<\infty$ implies $H_4(0)<\infty$, so Proposition~\ref{prop:global-accumulated-field} applies.

Fix $T>0$ and set $S=1+T$. By \eqref{eq:particle-energy-growth}, Proposition~\ref{prop:global-accumulated-field}, and Theorem~\ref{thm:sublinear-interactions},
\begin{equation}\label{eq:uniform-characteristic-energy}
a_z(t)\leq a_0(z)+CS,\qquad 0\leq t\leq T.
\end{equation}
On the set $\{a_z(t)\leq2a_0(z)\}$, the $K$-th moment contribution is at most $2^KH_K(0)$. On its complement, \eqref{eq:uniform-characteristic-energy} implies $a_z(t)\leq2CS$. Hence
\[
\int_{\{a_z(t)>2a_0(z)\}}a_z(t)^K\,d\mu(z)
\leq(2CS)^{K-2}\int_{\mathbb R^6}a_z(t)^2\,d\mu(z)
\leq C_KS^{K-2},
\]
where Lemma~\ref{lem:basic-energy} controls the second moment. The sum of the two contributions is uniform over $0\leq t\leq T$. Taking the supremum finishes the proof.
\end{proof}

The bounds above apply uniformly to smooth solutions with the same initial finite-energy and moment bounds. In Section~5 we will use them to pass from smooth solutions to global weak solutions and to recover conservation of the total energy.

\section{Global weak solutions and energy conservation}\label{sec:weak}

Fix $K>2$ and initial data as in Theorem~\ref{thm:main}. Write $a_0(x,v)=a(0,x,v)$. By a standard truncation and mollification argument, there are nonnegative functions $f_{0,n}\in C_c^\infty(\mathbb R^6)$ whose spatial supports stay a positive distance from $\xi_0$ and such that
\[
\|f_{0,n}\|_{L^1}\le\|f_0\|_{L^1},\qquad
\|f_{0,n}\|_{L^\infty}\le\|f_0\|_{L^\infty},\qquad
\iint_{\mathbb R^6}a_0^K|f_{0,n}-f_0|\,\dd x\dd v\longrightarrow0.
\]
In particular, the initial relative energy moments converge. The initial kinetic and plasma--charge interaction energies also converge. For the plasma self-interaction energy, velocity interpolation gives convergence of the initial densities in $L^{6/5}$, and the Hardy--Littlewood--Sobolev inequality gives convergence of the corresponding field energies. Thus the initial total energies converge. We keep the initial charge data $(\xi_0,\eta_0)$ fixed.

The classical existence theorem of Marchioro, Miot and Pulvirenti~\cite[Theorem~2]{MMP2011} gives a global classical solution $(f_n,\xi_n,\eta_n)$ for each approximating datum. Each solution uses the original Coulomb kernel. Its phase-space support is compact on bounded time intervals and stays separated from the charge there.

\subsection{Uniform estimates and the charge trajectory}\label{subsec:uniform-charge}

Set $\mathcal K(x)=x/|x|^3$ for $x\ne0$, and write
\[
\rho_n=\int_{\mathbb R^3}f_n\,\dd v,\qquad
E_n=\mathcal K*\rho_n,\qquad
F_n(t,x)=\mathcal K(x-\xi_n(t)).
\]
Let $a_n$ be the relative energy weight
\[
a_n(t,x,v)=\left(1+|v-\eta_n(t)|^2+\frac2{|x-\xi_n(t)|}\right)^{1/2},
\qquad H_{K,n}(t)=\iint_{\mathbb R^6}a_n^Kf_n\,\dd x\dd v.
\]
Write $\mathcal E_n$ for the total energy of $(f_n,\xi_n,\eta_n)$. For a time interval $I$, put
\[
B_{n,I}=\int_I\|\rho_n(t)\|_2^2\,\dd t,\quad
Q_{n,I}=\int_I|E_n(t,\xi_n(t))|\,\dd t,\quad
D_{n,I}=\int_I\int_{\mathbb R^3}\frac{\rho_n(t,x)}{|x-\xi_n(t)|^2}\,\dd x\dd t,
\]
and abbreviate $Q_{n,[0,T]}$ and $D_{n,[0,T]}$ by $Q_{n,T}$ and $D_{n,T}$.

After discarding finitely many indices, the convergence of the initial energies and moments gives uniform bounds for $\mathcal E_n(0)$ and $H_{K,n}(0)$. Conservation of mass, the $L^\infty$ norm, and total energy, together with Sections~\ref{sec:preliminaries}--\ref{sec:moment-growth}, yield
\begin{gather}
\sup_n\sup_{t\ge0}\left(\iint_{\mathbb R^6}|v|^2f_n\,\dd x\dd v+|\eta_n(t)|^2+\|E_n(t)\|_2^2\right)\le C,\label{eq:s5-uniform-energy}\\
\sup_n\int_0^\infty\left(\|\rho_n(t)\|_2^2+\|E_n(t)\|_6^2\right)\,\dd t\le C,\qquad
\sup_n(Q_{n,T}+D_{n,T})\le C(1+T)^{2/3},\label{eq:s5-uniform-field}\\
\sup_n\sup_{0\le t\le T}(1+H_{K,n}(t))\le C_K(1+T)^{K-2}.\label{eq:s5-uniform-moment}
\end{gather}
For $2<K\le4$, the exponent in the last estimate can be replaced by $\frac23(K-2)$. Here $C$ depends only on $\|f_0\|_{L^1}$, $\|f_0\|_{L^\infty}$, and $\mathcal E(0)$; $C_K$ may also depend on $K$ and $H_K(0)$. All these bounds are uniform in $n$.

Fix $T>0$. The charge trajectory is the first compactness issue, since the singularity in the plasma equation moves with $\xi_n(t)$. A common modulus of continuity for $\eta_n$ will locate this singularity in the limit. A diagonal subsequence will then give a solution on the whole time axis.

\begin{proposition}\label{prop:charge-compactness}
Let $\alpha_K=3(K-1)/(3K+1)$. For every $0\le s\le t\le T$,
\begin{equation}\label{eq:s5-charge-modulus}
|\eta_n(t)-\eta_n(s)|\le\int_s^t |E_n(\tau,\xi_n(\tau))|\,\dd\tau
\le C_{T,K}|t-s|^{\alpha_K}.
\end{equation}
Along a subsequence, $\eta_n\to\eta$ in $C([0,T];\R^3)$ and $\xi_n\to\xi$ in $C^1([0,T];\R^3)$. The limits have the prescribed initial position and velocity and satisfy $\dot\xi=\eta$. In addition, $\eta\in C^{0,\alpha_K}([0,T];\R^3)$ and $\xi\in C^{1,\alpha_K}([0,T];\R^3)$.
\end{proposition}

\begin{proof}
Let $I=[s,t]$ and $\delta=t-s$. By \eqref{eq:s5-uniform-field}, $B_{n,I}\le B$ for a fixed $B\ge1$. Choose $\delta_0\in(0,1]$ so that $\delta_0^2B\le c_0$, where $c_0$ is the uniform constant in Proposition~\ref{prop:local-charge-field}. For $\delta\le\delta_0$, that proposition gives $Q_{n,I}\le L_0$, uniformly in $n$ and $s$.

Label the trajectories at $s$ by $z$, use $\dd\mu_{n,s}(z)=f_n(s,z)\,\dd z$, and put $G_z=\int_I|E_n(\tau,X_{n,z}(\tau))|\,\dd\tau$. Lemma~\ref{lem:accumulated-plasma-field} gives $\int G_z^2\,\dd\mu_{n,s}\le C\delta B_{n,I}\le C\delta$. For $0<\delta\le\delta_*\le\delta_0$, take $P=\delta^{-3/(3K+1)}$, with $\delta_*$ small enough that $P\ge8(1+L_0)$. Let $A$ consist of the labels whose relative velocity exceeds $P$ somewhere on $I$, and let $m=\mu_{n,s}(A)$.

By \eqref{eq:relative-weight}, the weight along each trajectory is at most $a_n(s,z)+G_z+Q_{n,I}$ on $I$. Since $Q_{n,I}\le P/8$, a label in $A$ has $a_n(s,z)>P/4$ or $G_z>P/4$. Chebyshev's inequality, \eqref{eq:s5-uniform-moment}, and the bound on $\int G_z^2\,\dd\mu_{n,s}$ imply
\begin{equation}\label{eq:s5-high-mass}
m\le C_{T,K}(P^{-K}+\delta P^{-2})\le C_{T,K}P^{-K},
\qquad \delta P^{K-2}=\delta^{7/(3K+1)}\le1 .
\end{equation}
For the complementary labels, the velocities remain in a ball of radius $P$ about $\eta_n$. Their density $\rho^{\mathrm{low}}$ is bounded by $CP^3$ and is dominated by $\rho_n$. By \eqref{eq:coulomb-subpopulation}, its contribution to $Q_{n,I}$ is bounded by $C\delta P^{4/3}$.

To estimate the contribution of $A$, integrate \eqref{eq:radial-trajectory-bound} over those labels. This bounds their absolute field integral on the charge by their total inverse-square interaction. The endpoint relative velocity terms cost at most $C_{T,K}m^{1-1/K}$ by H\"older's inequality and \eqref{eq:s5-uniform-moment}. The field term along the plasma trajectories costs at most $C\delta^{1/2}m^{1/2}$ by Cauchy--Schwarz. The charge term is at most $L_0m$. Therefore \eqref{eq:s5-high-mass} yields
\[
Q_{n,I}\le C_{T,K}\bigl(\delta P^{4/3}+P^{-(K-1)}+\delta^{1/2}P^{-K/2}+L_0P^{-K}\bigr).
\]
The first two terms have power $\alpha_K$ in $\delta$. The exponents of the last two exceed $\alpha_K$ by $7/[2(3K+1)]$ and $3/(3K+1)$, respectively. This bounds $Q_{n,I}$ by $C_{T,K}\delta^{\alpha_K}$ for $\delta\le\delta_*$. For longer intervals use $Q_{n,I}\le Q_{n,T}$ from \eqref{eq:s5-uniform-field}. The first inequality in \eqref{eq:s5-charge-modulus} follows by integrating Newton's equation.

The velocities are uniformly bounded by \eqref{eq:s5-uniform-energy}, so Arzel\`a--Ascoli gives a uniformly convergent subsequence. Integrating $\dot\xi_n=\eta_n$ from $\xi_0$ yields convergence in $C^1$ and fixes the initial position. The limit velocity starts at $\eta_0$ and inherits the same H\"older modulus.
\end{proof}

\subsection{Compactness of the plasma and its field}\label{subsec:plasma-compactness}

A velocity cutoff lets us use local compactness in $(t,x)$. The mass in large positions and velocities is uniformly small, which lets us remove the cutoffs and recover the full density.

\begin{proposition}\label{prop:plasma-compactness}
Along a further subsequence, there is a nonnegative $f\in L^\infty(0,T;L^1(\R^6)\cap L^\infty(\R^6))$ with density $\rho=\int_{\R^3}f\,\dd v$ and field $E=\mathcal K*\rho$ such that
\begin{equation}\label{eq:s5-plasma-convergence}
\begin{gathered}
f_n\rightharpoonup^*f\quad\text{in }L^\infty((0,T)\times\R^6),\\
\rho_n\longrightarrow\rho\quad\text{in }L^p((0,T)\times\R^3)\quad(1\le p<2),\\
E_n\longrightarrow E\quad\text{in }L^2((0,T)\times\R^3).
\end{gathered}
\end{equation}
For almost every $t$, $\int_{\R^3}\rho(t,x)\,\dd x=\|f_0\|_{L^1}$. For every $\varphi\in C_c^\infty([0,T)\times\R^6)$,
\begin{equation}\label{eq:s5-regular-force}
\int_0^T\iint_{\R^6} f_n E_n\cdot\nabla_v\varphi\,\dd x\dd v\dd t
\longrightarrow
\int_0^T\iint_{\R^6} f E\cdot\nabla_v\varphi\,\dd x\dd v\dd t .
\end{equation}
\end{proposition}

\begin{proof}
First consider the spatial tails. A particle at $|X_{n,z}(t)|>R$ either started at $|x|>R/2$ or moved more than $R/2$. In the latter case the integral of its speed is greater than $R/2$. By the flow change of variables,
\[
\int_{|x|>R}\rho_n(t,x)\,\dd x
\le \iint_{|x|>R/2} f_{0,n}(x,v)\,\dd x\dd v
 +\frac2R\int_0^t\iint_{\R^6}|v| f_n(\tau,x,v)\,\dd x\dd v\dd\tau .
\]
The first term tends to zero uniformly in $n$ by the $L^1$ convergence of $f_{0,n}$ and tightness of any finite number of the initial densities. The second is at most $CT/R$ by the mass and kinetic energy bounds. For $R\ge1$ large enough, $|v|>R$ implies $a_n\ge |v|/2$, since $\eta_n$ is uniformly bounded. Hence
\begin{equation}\label{eq:s5-tightness}
\lim_{R\to\infty}\sup_{n,\,0\le t\le T}\int_{|x|>R}\rho_n(t,x)\,\dd x=0,\qquad
\sup_{n,\,0\le t\le T}\iint_{|v|>R} f_n(t,x,v)\,\dd x\dd v\le C_{T,K}R^{-K}.
\end{equation}
The last estimate follows from \eqref{eq:s5-uniform-moment}.

By weak-$*$ compactness of bounded sequences in $L^\infty$, we can extract a limit $f$ on the slab. Nonnegative tests give $f\ge0$. Using spatial and velocity cutoffs increasing to one also gives $\int_0^T\iint f\,\dd x\dd v\dd t\le T\|f_0\|_{L^1}$.

We use the following local form of velocity averaging with the fixed exponent $p=4/3$. If $u_n,G_n,h_n$ are bounded in $L^{4/3}_{\mathrm{loc}}$ and
\[
(\partial_t+v\cdot\nabla_x)u_n=\operatorname{div}_v G_n+h_n,
\]
then, for every $\psi\in C_c^\infty(\R^3)$, the averages $\int u_n(t,x,v)\psi(v)\,\dd v$ are relatively compact in $L^{4/3}_{\mathrm{loc}}$ in $(t,x)$. Time is restricted to the interior of the slab. This follows from DiPerna, Lions and Meyer's theorem~\cite[Section~III, Theorem~5]{DLM1991}. After localization in $(t,x,v)$, the derivatives of the cutoffs add only locally bounded $L^{4/3}$ terms. The localized right-hand side can be written as $(1-\Delta_v)^{1/2}\widetilde G_n$ with $\widetilde G_n$ bounded in $L^{4/3}$, by the $L^{4/3}$ multiplier bounds for $(1-\Delta_v)^{-1/2}$ and $(1-\Delta_v)^{-1/2}\partial_{v_j}$. The theorem gives positive Besov regularity in $(t,x)$ and hence the stated local compactness.

Let $0\le\chi_R\le1$ be smooth, equal to one on $|v|\le R$ and zero for $|v|\ge2R$. In our case $u_n=\chi_R f_n$ satisfies
\[
(\partial_t+v\cdot\nabla_x)u_n
=-\operatorname{div}_v\bigl((E_n+F_n)u_n\bigr)
+f_n(E_n+F_n)\cdot\nabla_v\chi_R .
\]
The field $E_n$ is bounded in $L^\infty_tL^2_x$ by \eqref{eq:s5-uniform-energy}, and hence in $L^{4/3}_{\mathrm{loc}}$ on the slab. For any bounded spatial set $\Omega$,
\[
\int_\Omega |F_n(t,x)|^{4/3}\,\dd x
\leq \int_{|y|\le1}|y|^{-8/3}\,\dd y+|\Omega|<\infty.
\]
This bound is independent of $t$ and of the center $\xi_n(t)$, since $8/3<3$. Thus $F_n$ is uniformly bounded in $L^{4/3}_{\mathrm{loc}}$. Since $f_n$ is uniformly bounded, both $(E_n+F_n)u_n$ and $f_n(E_n+F_n)\cdot\nabla_v\chi_R$ are bounded in the local spaces required above.

Choose $\psi$ equal to one on the velocity support of $\chi_R$. The average $\int \chi_R f_n\,\dd v$ is relatively compact locally, and weak-$*$ convergence identifies its only possible limit as $\int\chi_R f\,\dd v$. In the same way, for each compactly supported smooth velocity test $\psi$, $\int\psi f_n\,\dd v$ converges locally in $L^1_{t,x}$ to $\int\psi f\,\dd v$. One may first choose $\chi_R=1$ on the support of this test.

Fix $h>0$ and $L<\infty$. On $[h,T-h]\times B_L$, the velocity tail in \eqref{eq:s5-tightness} allows the velocity cutoff to be removed, and the local velocity averages give strong $L^1$ convergence of the full densities. Next let $L\to\infty$. The spatial tails are uniformly small by \eqref{eq:s5-tightness} for $\rho_n$ and by integrability for $\rho$. Finally let $h\downarrow0$. The two time strips of width $h$ have total $\rho_n$ mass at most $2\|f_0\|_{L^1}h$, and the corresponding mass of $\rho$ also tends to zero. These three steps give strong $L^1$ convergence on $(0,T)\times\mathbb R^3$. Since the exact masses of the approximations tend to $\|f_0\|_{L^1}$, the functions $t\mapsto \int\rho_n(t,x)\,\dd x$ identify the limit mass as $\|f_0\|_{L^1}$ almost everywhere.

The uniform $L^2_{t,x}$ bound in \eqref{eq:s5-uniform-field} gives the same bound for $\rho$ by weak lower semicontinuity. Interpolating with strong $L^1$ convergence yields strong $L^p_{t,x}$ convergence for $1<p<2$. In particular, by \eqref{eq:coulomb-HLS},
\[
\|E_n-E\|_{L^{6/5}(0,T;L^2_x)}\le C\|\rho_n-\rho\|_{L^{6/5}((0,T)\times\R^3)}\longrightarrow0.
\]
This convergence identifies the weak-$*$ limit in $L^\infty_tL^2_x$ supplied by the energy bound. Interpolating in time between $6/5$ and $\infty$ gives strong $L^2_{t,x}$ convergence, with
\[
\|E_n-E\|_{L^2_tL^2_x}
\le \|E_n-E\|_{L^{6/5}_tL^2_x}^{3/5}\|E_n-E\|_{L^\infty_tL^2_x}^{2/5}.
\]

For the regular force, write $f_n E_n-fE=f_n(E_n-E)+(f_n-f)E$. After testing with $\nabla_v\varphi$, the first term tends to zero by strong $L^2$ convergence on the compact support of the test and the uniform $L^\infty$ bound on $f_n$. The second tends to zero by weak-$*$ convergence, because $E\cdot\nabla_v\varphi$ is integrable there.
\end{proof}

\subsection{Passage to the singular forces}\label{subsec:singular-forces}

The two forces involving the charge use the inverse-square kernel. We first control the part of their integrals lying in a tube of small radius around the charge trajectory. The control must hold uniformly for the approximating solutions.

\begin{proposition}\label{prop:shrinking-tube}
There is a constant $C_{T,K}$ such that for every $\varepsilon>0$,
\begin{equation}\label{eq:s5-tube}
\sup_n\int_0^T\int_{|x-\xi_n(t)|<\varepsilon}
\frac{\rho_n(t,x)}{|x-\xi_n(t)|^2}\,\dd x\dd t
\le C_{T,K}\varepsilon^{1/2}.
\end{equation}
\end{proposition}

\begin{proof}
Label the trajectories at time zero with measure $\dd\mu_{n,0}(z)=f_{0,n}(z)\,\dd z$. Write $G_z$ for the integral of $|E_n|$ along the trajectory of $z$ up to time $T$. By Lemma~\ref{lem:accumulated-plasma-field} and \eqref{eq:s5-uniform-field}, $\int G_z^2\,\dd\mu_{n,0}\le CT$. Let $A_{n,\varepsilon}$ be the set of labels whose trajectories enter the tube on $[0,T]$. It is measurable by continuity of the trajectories. For such a label there is a time when the relative weight is at least $(2/\varepsilon)^{1/2}$. Yet \eqref{eq:relative-weight} bounds the relative weight throughout the interval by $a_0(z)+G_z+Q_{n,T}$.

Choose $\varepsilon_T\le1$ so that $Q_{n,T}\le(2/\varepsilon_T)^{1/2}/2$ for every $n$, using \eqref{eq:s5-uniform-field}. Any label in $A_{n,\varepsilon}$ then satisfies $a_0(z)\ge c\varepsilon^{-1/2}$ or $G_z\ge c\varepsilon^{-1/2}$. Chebyshev's inequality and the bounds on the two moments give
\[
m_{n,\varepsilon}:=\mu_{n,0}(A_{n,\varepsilon})
\le C_{T,K}(\varepsilon^{K/2}+\varepsilon)\le C_{T,K}\varepsilon .
\]
By the flow change of variables the integral in \eqref{eq:s5-tube} is bounded by
\[
\int_{A_{n,\varepsilon}}\int_0^T\frac{\dd t}{|X_{n,z}(t)-\xi_n(t)|^2}\,\dd\mu_{n,0}(z).
\]
Use \eqref{eq:radial-trajectory-bound} on these labels. The endpoint terms cost at most $C_{T,K}m_{n,\varepsilon}^{1-1/K}$ by \eqref{eq:s5-uniform-moment}, and the plasma-field term costs at most $(CT)^{1/2}m_{n,\varepsilon}^{1/2}$. The charge-field term is $Q_{n,T}m_{n,\varepsilon}$. Since $1-1/K>1/2$ and $Q_{n,T}$ is uniformly bounded, each contribution is bounded by $C_{T,K}\varepsilon^{1/2}$ for $\varepsilon\le\varepsilon_T$. For larger radii the uniform bound on $D_{n,T}$ in \eqref{eq:s5-uniform-field} gives the same estimate after increasing $C_{T,K}$.
\end{proof}

We can now cut off either singular force at radius $\varepsilon$ and control the error as $\varepsilon$ tends to zero. The same estimate also defines the Coulomb integral along the limiting path.

\begin{proposition}\label{prop:singular-convergence}
For almost every $t$, the Coulomb integral defining $E(t,\xi(t))$ is absolutely convergent. Moreover,
\begin{equation}\label{eq:s5-charge-field}
E_n(\cdot,\xi_n(\cdot))\longrightarrow E(\cdot,\xi(\cdot))
\quad\text{in }L^1(0,T;\R^3).
\end{equation}
With $F(t,x)=\mathcal K(x-\xi(t))$, for every $\varphi\in C_c^\infty([0,T)\times\R^6)$,
\begin{equation}\label{eq:s5-point-force}
\int_0^T\iint_{\R^6} f_nF_n\cdot\nabla_v\varphi\,\dd x\dd v\dd t
\longrightarrow
\int_0^T\iint_{\R^6} fF\cdot\nabla_v\varphi\,\dd x\dd v\dd t .
\end{equation}
The triple $(f,\xi,\eta)$ is a weak solution on $[0,T]$ in the sense of Definition~\ref{def:weak-solution}.
\end{proposition}

\begin{proof}
To pass the tube bound to the limit density, For fixed $\varepsilon,N>0$, take a smooth $\chi\colon[0,\infty)\to[0,1]$ equal to one on $[0,1]$ and zero on $[2,\infty)$. Use the weights
\[
\chi\!\left(\frac{|x-\xi_n(t)|}{\varepsilon}\right)\min\{N,|x-\xi_n(t)|^{-2}\}.
\]
The minimum has value $N$ at the center. These are translates of a bounded, uniformly continuous kernel. By Proposition~\ref{prop:charge-compactness} they converge uniformly to the weights with $\xi$ in place of $\xi_n$. Strong $L^1_{t,x}$ convergence of the densities passes their integrals to the limit, using the uniform mass bound for the error from translation. The integrals are bounded by \eqref{eq:s5-tube} at radius $2\varepsilon$. Letting $N\uparrow\infty$ by monotone convergence gives
\begin{equation}\label{eq:s5-limit-tube}
\int_0^T\int_{|x-\xi(t)|<\varepsilon}\frac{\rho(t,x)}{|x-\xi(t)|^2}\,\dd x\dd t
\le C_{T,K}\varepsilon^{1/2},\qquad \varepsilon>0 .
\end{equation}
Outside this tube the kernel is bounded by $\varepsilon^{-2}$, and the density is integrable. By Tonelli's theorem the integral $E(t,\xi(t))=\int_{\R^3}\mathcal K(\xi(t)-x)\rho(t,x)\,\dd x$ is absolutely convergent for almost every $t$. These estimates also give uniform integrability in time. On a measurable time set $A\subset(0,T)$, the inverse-square integrals are bounded by $C_{T,K}\varepsilon^{1/2}+\|f_0\|_{L^1}\varepsilon^{-2}|A|$, uniformly for the approximations and the limit.

For both force limits use a smooth radial $\vartheta_\varepsilon$ valued in $[0,1]$, equal to zero on $|x|\le\varepsilon$ and to one on $|x|\ge2\varepsilon$, and put $\mathcal K_\varepsilon=\vartheta_\varepsilon\mathcal K$. For fixed $\varepsilon$ this is bounded and globally Lipschitz. Denote by $e_{n,\varepsilon}$ and $e_\varepsilon$ the fields at $\xi_n,\xi$ computed with $\mathcal K_\varepsilon$ and densities $\rho_n,\rho$, respectively. Then
\[
\begin{aligned}
\|e_{n,\varepsilon}-e_\varepsilon\|_{L^1(0,T)}
&\le \|\mathcal K_\varepsilon\|_\infty\|\rho_n-\rho\|_{L^1((0,T)\times\R^3)}\\
&\quad+T\|f_0\|_{L^1}\,\operatorname{Lip}(\mathcal K_\varepsilon)\|\xi_n-\xi\|_{C([0,T])}
\longrightarrow0 .
\end{aligned}
\]
Since $|\mathcal K-\mathcal K_\varepsilon|\le |x|^{-2}\mathbf1_{\{|x|<2\varepsilon\}}$, the errors in removing the cutoff are bounded in $L^1(0,T)$ by $C_{T,K}\varepsilon^{1/2}$. This follows from \eqref{eq:s5-tube} for the approximations and \eqref{eq:s5-limit-tube} for the limit. We first let $n\to\infty$ at fixed radius and then let $\varepsilon\downarrow0$.

In the charge equation, the time integral of the force converges uniformly by \eqref{eq:s5-charge-field}. Passing to the limit in the two Newton equations gives, for every $t\in[0,T]$,
\[
\eta(t)=\eta_0+\int_0^t E(s,\xi(s))\,\dd s,\qquad
\xi(t)=\xi_0+\int_0^t\eta(s)\,\dd s.
\]
In particular, $\eta$ belongs to $W^{1,1}(0,T;\R^3)$.

In the Vlasov equation, $\mathcal K_\varepsilon(x-\xi_n(t))$ converges uniformly to $\mathcal K_\varepsilon(x-\xi(t))$ for fixed $\varepsilon$. Multiplication by $\nabla_v\varphi$ restricts the coefficient to a common compact set. The truncated terms therefore converge by weak-$*$ convergence of $f_n$. The error at the approximating level is at most
\[
\|\nabla_v\varphi\|_\infty\int_0^T\int_{|x-\xi_n(t)|<2\varepsilon}
\frac{\rho_n(t,x)}{|x-\xi_n(t)|^2}\,\dd x\dd t .
\]
This is bounded by $C_{\varphi,T,K}\varepsilon^{1/2}$. The limit error obeys the same bound by \eqref{eq:s5-limit-tube}. Removing the cutoff gives \eqref{eq:s5-point-force}. The free transport terms pass by weak-$*$ convergence, the initial term by $f_{0,n}\to f_0$ in $L^1$, and the regular force by \eqref{eq:s5-regular-force}.

These limits give the Vlasov identity \eqref{eq:weak-vlasov} with the prescribed initial datum. The full inverse-square integral is finite by \eqref{eq:s5-limit-tube} and the mass bound, as required by \eqref{eq:weak-singular-integrability}. Nonnegativity and the bounds on $f$ follow from Proposition~\ref{prop:plasma-compactness}. Proposition~\ref{prop:charge-compactness} and the integral Newton equations give the time regularity and initial conditions in Definition~\ref{def:weak-solution}.
\end{proof}

\subsection{Moment bounds and energy conservation}\label{subsec:limit-conservation}

The density and singular interaction estimates survive on the finite slab. For the density, use weak lower semicontinuity in $L^2((0,T)\times\R^3)$ and \eqref{eq:coulomb-HLS} with $p=2$. This gives
\[
\int_0^T\bigl(\|\rho(t)\|_2^2+\|E(t)\|_6^2\bigr)\,\dd t\le C
\]
with $C$ independent of $T$. Let $Q_T,D_T$ be the integrals for the limit defined as for $Q_{n,T},D_{n,T}$. Strong convergence in \eqref{eq:s5-charge-field} gives $Q_{n,T}\to Q_T$. For $D_T$ use the bounded Lipschitz kernel $\min\{N,|x|^{-2}\}$ with value $N$ at zero. Strong $L^1$ convergence of the densities and uniform convergence of the paths pass the truncated integral to the limit. Add $Q_{n,T}$ before taking $n\to\infty$ and then let $N\uparrow\infty$ by monotone convergence. The bound in \eqref{eq:s5-uniform-field} gives $Q_T+D_T\le C(1+T)^{2/3}$.

The relative weight changes with both the position and velocity of the charge. We truncate the weight in phase space before using weak convergence. The resulting moment estimate will control the velocity tails of the plasma kinetic energy.

\begin{proposition}\label{prop:limit-moments}
On $(0,T)$ the limiting solution satisfies
\begin{equation*}
\operatorname*{ess\,sup}_{0<t<T}(1+H_K(t))\le C_K(1+T)^{K-2}.
\end{equation*}
For $2<K\le4$, the exponent on the right can be taken to be $\frac23(K-2)$. The charge satisfies $\eta\in C^{0,\alpha_K}([0,T];\R^3)$ and $\xi\in C^{1,\alpha_K}([0,T];\R^3)$, with $\alpha_K=3(K-1)/(3K+1)$. If $K\ge4$, then $\eta$ is Lipschitz on $[0,T]$.
\end{proposition}

\begin{proof}
Let $A_{K,T}$ be the uniform bound in \eqref{eq:s5-uniform-moment}, with the improved exponent from Proposition~\ref{prop:low-moments} when $2<K\le4$. Choose nonnegative cutoffs $\chi_R^x,\chi_R^v$ that are one on balls of radius $R$, zero outside balls of radius $2R$, and increase to one as $R$ increases. For fixed $N,R,\varepsilon>0$, the bounded weight
\[
w_n(t,x,v)=\chi_R^x(x)\chi_R^v(v)
 \min\!\left\{N,\left(1+|v-\eta_n(t)|^2+
               \frac{2}{|x-\xi_n(t)|+\varepsilon}\right)^{K/2}\right\}
\]
has compact phase-space support and satisfies $0\le w_n\le a_n^K$. It converges uniformly to the analogous weight $w$ with $\xi,\eta$ in place of $\xi_n,\eta_n$. For any nonnegative $\psi\in C_c^\infty(0,T)$, weak-$*$ convergence and the uniform moment bound therefore give
\[
\int_0^T\psi(t)\left(1+\iint_{\R^6}w f\,\dd x\dd v\right)\dd t
\le A_{K,T}\int_0^T\psi(t)\,\dd t.
\]
Here replacing $w_n$ by $w$ costs at most $T\|f_0\|_{L^1}\|\psi\|_\infty\|w_n-w\|_\infty$. Let $R\uparrow\infty$, $\varepsilon\downarrow0$, and then $N\uparrow\infty$. Monotone convergence applies at each step. The set $x=\xi(t)$ is null for the phase-space density. Since the inequality holds for each nonnegative time test, $1+H_K(t)\le A_{K,T}$ for almost every $t$.

The H\"older bound follows from the uniform estimate and convergence in Proposition~\ref{prop:charge-compactness}. For $K\ge4$, we have $a^K\ge a^4\ge4|x-\xi(t)|^{-2}$. Thus
\[
|E(t,\xi(t))|\le \int_{\R^3}\frac{\rho(t,x)}{|x-\xi(t)|^2}\,\dd x
\le \tfrac14H_K(t)
\quad\text{for almost every }t .
\]
The right side is essentially bounded on $(0,T)$. In the integral Newton equation this gives the Lipschitz estimate on $[0,T]$.
\end{proof}

We pass the four energy terms to the limit separately. The plasma kinetic term uses the moment bound just obtained. The two interaction terms require different estimates for the field and the charge.

\begin{proposition}\label{prop:total-energy}
The four energy terms of $(f_n,\xi_n,\eta_n)$ converge in $L^1(0,T)$ to the corresponding terms of $(f,\xi,\eta)$. In particular, $\mathcal E(t)=\mathcal E(0)$ for almost every $t\in(0,T)$.
\end{proposition}

\begin{proof}
\textit{Plasma kinetic energy.}
Since $|v|\le a_n+|\eta_n|$, the uniform moment and energy estimates bound $\iint |v|^K f_n\,\dd x\dd v$ on $[0,T]$. By Proposition~\ref{prop:limit-moments}, the same bound holds for $f$ at almost every time. For $R\ge1$, it follows that
\begin{equation}\label{eq:s5-kinetic-tail}
\sup_{n,t\le T}\iint_{|v|>R}|v|^2 f_n\,\dd x\dd v
+\operatorname*{ess\,sup}_{0<t<T}\iint_{|v|>R}|v|^2 f\,\dd x\dd v
\le C_{T,K}R^{2-K}.
\end{equation}
For fixed $R$, use the velocity test $|v|^2\chi_R(v)$ from the proof of Proposition~\ref{prop:plasma-compactness}. It gives local strong $L^1_{t,x}$ convergence of $k_{n,R}=\int |v|^2\chi_R f_n\,\dd v$ to $k_R=\int |v|^2\chi_R f\,\dd v$ along the same subsequence. The inequalities $0\le k_{n,R}\le4R^2\rho_n$ and $0\le k_R\le4R^2\rho$ control the spatial tails by \eqref{eq:s5-tightness} and the integrability of $\rho$. They also control the two time strips of width $h$ by $8R^2\|f_0\|_{L^1}h$. This extends the convergence to all of $(0,T)\times\R^3$. Now \eqref{eq:s5-kinetic-tail} removes the velocity cutoff with an error tending to zero in $L^1_t$, uniformly in $n$ and also for $f$. Integrating in space gives the desired $L^1_t$ kinetic energy convergence.

\textit{Charge kinetic energy.}
The uniform convergence of $\eta_n$ on $[0,T]$ gives uniform convergence of $\frac12|\eta_n|^2$.

\textit{Plasma self-interaction energy.}
The same velocity splitting as in the proof of Lemma~\ref{lem:basic-energy}, together with the mass and second-moment bounds, gives $\rho_n,\rho\in L^{5/3}_x\cap L^1_x$ and hence $L^{6/5}_x$ for almost every time. The field-energy identity established in that proof therefore applies to both $\rho_n$ and $\rho$. Strong convergence of the fields in \eqref{eq:s5-plasma-convergence} gives
\[
\int_0^T\bigl|\|E_n(t)\|_2^2-\|E(t)\|_2^2\bigr|\,\dd t
\le \bigl(\|E_n\|_{L^2_{t,x}}+\|E\|_{L^2_{t,x}}\bigr)\|E_n-E\|_{L^2_{t,x}}
\longrightarrow0.
\]
Multiplication by $(8\pi)^{-1}$ gives convergence of the self-interaction terms in $L^1(0,T)$.

\textit{Plasma--charge interaction energy.}
Use the cutoff $\vartheta_\varepsilon$ from Proposition~\ref{prop:singular-convergence}. For fixed $\varepsilon$, the function $x\mapsto\vartheta_\varepsilon(x)|x|^{-1}$ is bounded and Lipschitz. Strong $L^1$ convergence of the densities and uniform convergence of the centers give convergence of the corresponding truncated energies in $L^1(0,T)$. In the removed region, $|x-\xi_n(t)|^{-1}\le2\varepsilon|x-\xi_n(t)|^{-2}$. Hence \eqref{eq:s5-tube} gives an $L^1(0,T)$ error at most $C_{T,K}\varepsilon^{3/2}$, and \eqref{eq:s5-limit-tube} gives the same bound for the limit energy. Sending $\varepsilon$ to zero gives convergence of the mixed term.

Consequently $\mathcal E_n\to\mathcal E$ in $L^1(0,T)$. For each $n$ the function $\mathcal E_n(t)$ is equal to the constant $\mathcal E_n(0)$. The convergence of these initial energies in Section~\ref{subsec:uniform-charge} identifies $\mathcal E$ with the constant $\mathcal E(0)$ almost everywhere on the slab.
\end{proof}

Apply the construction on $[0,1]$, then on $[0,2]$ using a subsequence of the first sequence, and continue with nested subsequences. Take the $j$th index of the $j$th subsequence, each listed in increasing order. On every fixed $[0,m]$ the diagonal sequence has the convergences above. Since the restrictions of the limits agree on overlapping intervals, they define a single triple $(f,\xi,\eta)$ on $[0,\infty)$. Proposition~\ref{prop:singular-convergence} gives Definition~\ref{def:weak-solution} on every finite interval, with the prescribed initial data.

For any $T>0$, restrict the convergences from an integer interval containing $[0,T]$ and apply the estimates on the shorter slab. The passage of $Q_{n,T}+D_{n,T}$ gives \eqref{eq:main-cumulative-interaction}. Proposition~\ref{prop:limit-moments} gives \eqref{eq:main-polynomial-moment}, the improved low-order bound, and the charge regularity \eqref{eq:main-holder}, including the Lipschitz conclusion for $K\ge4$. The density and field integral is bounded independently of the slab length, so monotone convergence gives \eqref{eq:main-global-density}. The exceptional times for the energy identity form a null set on each integer interval by Proposition~\ref{prop:total-energy}. Their countable union is null, which proves \eqref{eq:main-energy-conservation} on the whole half-line.

\quad\\

\begin{samepage}

\nopagebreak[4]
\noindent\textbf{Acknowledgements.} The author would like to thank the anonymous referees for their helpful comments and suggestions. The author is also grateful to colleagues and friends at BIT and HUST for many useful discussions and suggestions related to this work.
\quad\\

\nopagebreak[4]
\noindent\textbf{Data/Code Availability.} The manuscript has no associated data or code.
\quad\\

\nopagebreak[4]
\noindent\textbf{Conflict of interest.} The author declares that he has no conflict of interest.
\quad\\

\nopagebreak[4]
\noindent\textbf{AI-use disclosure.} OpenAI's ChatGPT assisted with drafting and editing the English exposition, organizing the manuscript, bibliographic checks, and exploratory mathematical discussion. The author has reviewed and edited all AI-assisted material and takes full responsibility for the content of this manuscript.

\end{samepage}

\renewcommand{\bibliofont}{\footnotesize\linespread{1.02}\selectfont}

\end{document}